\documentclass[12pt]{article}
\usepackage{amsmath,amsthm,amssymb,amsfonts}
\usepackage{cite}
\usepackage{xcolor}
\usepackage{comment}
\usepackage{bm} 
\usepackage{float}
\usepackage[pdftex]{hyperref}
\usepackage{units}
\allowdisplaybreaks

\newcommand\indicator{{\mathbb I}}
\newcommand\vecm{{\bf{m}}}
\newcommand\vecs{{\bf{s}}}
\newcommand\inTwo{\buildrel {{\cal L}_2} \over  \longrightarrow}

\newcommand\bfsigma{{\mathbf\Sigma}}
\newcommand\bfo{{\bf o}}
\newcommand\bfO{{\bf O}}
\newcommand\veca{{\bf a}}
\newcommand\vecb{{\bf b}}
\newcommand\vece{{\bf e}}
\newcommand\vecQ{{\bf Q}}
\newcommand\vecv{{\bf v}}
\newcommand\vecX{{\bf X}}
\newcommand\matA{{\bf A}}
\newcommand\matB{{\bf B}}
\newcommand\matC{{\bf C}}
\newcommand\matI{{\bf I}}
\newcommand{\iu}{\mathrm{i}\mkern1mu} 
\newcommand\matM{{\bf M}}

\newcommand{\inL}[1]{\buildrel \mathcal{L}_{#1} \over\longrightarrow}
\newcommand{\field}{{\mathbb F}}
\newcommand{\normal}{{\mathcal{N}}}

\newtheorem{theorem}{Theorem}[section]
\newtheorem{lemma}{Lemma}[section]

\newtheorem{cor}{Corollary}[section]
\newtheorem{remark}{Remark}[section]
\newtheorem{example}{Example}[section]

\newcommand\polya{{P\' olya}}
\newcommand\convD{{\buildrel {\cal D} \over \longrightarrow}}
\newcommand\given{\, \vert \, } 
 
\newcommand\almostsure{\buildrel a.s. \over  \longrightarrow}
\newcommand\inprob{\buildrel {\mathbb P} \over  \longrightarrow}

\newcommand\Prob{{\mathbb P}}
\newcommand\E{{\mathbb E}}
\newcommand\V{{\mathbb V}{\rm ar}}
\newcommand\Cov{{\mathbb C}{\rm ov}}

\begin{document}
\begin{center}
{\bf \Huge On affine multi-color urns grown under multiple drawing}

\bigskip

Joshua Sparks\footnote{Department of Statistics,
The George Washington University, Washington, D.C. 20052, U.S.A. Email: josparks@gwu.edu}
\quad
Markus Kuba\footnote{Department of Applied Mathematics and Physics, 
FH Technikum Wien, Vienna, Austria.
Email: kuba@technikum-wien.at}
\quad
Srinivasan Balaji\footnote{Department of Statistics,
The George Washington University, Washington, D.C. 20052, U.S.A. Email: balaji@gwu.edu}
\quad 
Hosam Mahmoud\footnote{Department of Statistics,
The George Washington University, Washington, D.C. 20052, U.S.A. Email: hosam@gwu.edu}
\quad 
\date{\today}
\end{center}
\section*{Abstract}
{Early investigation of \polya\ urns considered drawing balls one at a time.
In the last two decades, several authors considered multiple drawing in each step, but
mostly for schemes on two colors. In this manuscript, we consider multiple drawing
from urns of balls of multiple colors, formulating asymptotic theory for 
specific urn classes and addressing more applications.

The class we consider is affine and tenable, built around a ``core'' square matrix. An index 
for the drawing schema is derived from the eigenvalues of the core. We identify three
regimes: small\hbox{-,} critical\hbox{-,} and large-index. In the small-index regime, we find an asymptotic 
Gaussian law. In the critical-index regime, we also find an asymptotic Gaussian law, albeit  
a difference in the scale factor, which 
involves logarithmic terms. 

In both of these regimes, we have 
explicit forms for the structure of the mean and the covariance matrix of the composition vector (both exact and asymptotic). In all three regimes we have strong laws.

\bigskip\noindent
{\bf Keywords:} \polya\ urn, combinatorial probability, limiting distribution, 
multivariate central limit theorem, multivariate martingale.

\bigskip\noindent
{\bf MSC:} {60F05, 
                           60C05, 
                           60G46}  
\section{Introduction}
\polya\ urns are popular modeling tools, as they capture the dynamics of
many real-world applications involving change over time. The book~\cite{Mah2009} 
lists numerous applications in informatics and biosciences; see also \cite{Bala,JohnsonKotz} for a wider
perspective. 

In the early models,  single balls are drawn one at a time.
In the last two decades, several authors generalized the models to schemes
evolving on multiple drawing; see~\cite{Bose,Chen1,Chen2,Crimaldi,Panholzer,
Kuba,Lasmar,Mah2013}. 
However, most of these investigations deal only with urn schemes 
on two colors. So far, explicit expressions and asymptotic expansions for the expected value, the variance and covariances 
{and} higher moments have proven to be elusive in the 
{multi-color} case.

In this manuscript, we present a more comprehensive 
approach in which there are multiple colors in the urn, and we sample several balls 
in each drawing. Depending on the composition of the sample, action is taken to 
place balls in the urn.  We ought to mention that this is a new trend in investigating
multi-color urns; little exists in the literature beyond the study of urn
schemes on an infinite number of colors, where the authors imposed restrictions
on the number of colors to add after each drawing~\cite{Ban,Janson2019}.

As a modeling tool, multi-color multiple drawing urns have their niche in applications
to capture the dynamics in systems that internally have multiple interacting parts. 
For instance, the Maki-Thompson rumor spreading model~\cite{Maki} can be modeled by
an urn scheme on three colors (representing ignorants, spreaders and stiflers of the rumor) experiencing
pair interactions{, such as social visits and phone calls (represented as draws of 
pairs of balls at a time).   Another application is to produce a profile of
the nodes in random hypergraphs see~\cite{Sparks1}, Chapter 6, and its prelude~\cite{Sparks2}.
\section{The mechanics of a multi-color multiple drawing urn scheme}
We first describe a model of multi-color multiple drawing urn qualitatively. 
We consider an urn growing on $k$ colors, which we label 
 $1,\ldots, k$. 
Initially, the urn contains a positive number of balls of colors $i\in [k]$ (some of the colors maybe initially absent).\footnote{The notation $[k]$ stands for the set $\{1, 2, \ldots, k\}$.}
At each discrete time step, a sample of balls is drawn out of the urn (the sample is taken without
replacement).\footnote{We can develop a model based on sampling with replacement.
The proof techniques and results for such a sampling scheme are very similar
to those that appear in a sampling scheme without replacement. To avoid distracting 
the reader, we do not discuss schemes under sampling with replacement in great detail.
We devote Section~\ref{Sec:rep} to outline results under sampling
with replacement.  
We refer a reader interested in more details to~\cite{Sparks1}.} 
Note that we speak of the entire sample as one draw. After the sample is collected, 
we put it back in the urn. Depending on what appears in the sample,
we add balls of various colors. Such models are to be called multi-color 
multiple drawing schemes.

Many such schemes can emerge upon specifying a particular 
growth algorithm (how many balls to draw, what to add, etc.). 

A sample of size $s$ is a {\em partition} of the number $s$ into nonnegative integers. Each part of the partition corresponds to the number of balls 
{in the sample} from a particular color after drawing the sample.
In other words, we view a sample as a row vector~$\vecs$ with nonnegative components
$s_1, \ldots, s_k$. The components of $\vecs$ are the solutions of the Diophantine equation
$$s_1 + \cdots + s_k = s$$
in non-negative integers. There are $s +k-1\choose s$ such solutions. 

In a very general model, when the sample $\vecs$ is drawn, we add $m_{\vecs,j}$ balls of color $j$, for $j=1, \ldots, k$, where $m_{\vecs,j}$ can be a  random variable.
We only consider the case of fixed additions: $m_{\vecs, j}\in \mathbb Z$,
the class of urn schemes where 
the sample size (in each drawing) is a positive fixed number $s$ (not changing over time), 
and the total number of balls
added after each drawing is also a fixed number $b$. Such an urn is said to 
be {\em balanced}, and $b$ is called the {\em balance factor}.
Further specification will ensue while crafting out conditions for linear recurrence.
 
We canonically rank the various samples through an application of { Algorithm} L in Knuth~\cite{Knuth} {called \emph{reverse lexicographic order}}, see \cite{Konzem}. In this scenario, the sample $(x_1, \ldots, x_k)$ precedes the sample 
$(y_1, \ldots, y_k)$, if for some $r \in [k]$, we have $x_i = y_i$ for 
$i=1, \ldots, r-1$, and $x_r > y_r$. For example, drawing samples of size $s=3$ from an urn
on $k=3$ colors, in this modification of the canonical ranking, 
the sample $(3,0,0)$ precedes the sample $(2,1,0)$, which in turn precedes the sample $(2, 0, 1)$.

Our interest is in urns that can endure indefinite drawing. That is,
no matter which stochastic path is followed, we can draw ad infinitum and
are able to execute the addition 
{rules} after each drawing. 

We call such a scheme
{\em tenable}. For instance, if all the numbers $m_{\vecs, j}$, for all feasible 
$\vecs$ and $j=1, \ldots, k$, are positive, the scheme is tenable.  
\section{The $(k,s,b)$-urn}
Consider a 
$k$-color
multiple drawing tenable balanced urn scheme grown
by taking samples of a fixed size $s\ge 1$, with balanced factor $b \ge 1$.
We call such a class $(k,s,b)$-{\em urns}.  
Let $X_{n,i}$ be the number of balls in the urn of color $i\in [k]$ after $n\ge 0$ draws, and let $\vecX_n$ be the {\em row} vector with these components. We call $\vecX_n$ the {\em composition vector}.
Thus, initially, there are $X_{0,i}\ge 0$ balls of color $i$, for $i=1,\ldots, k$.
Since the scheme is tenable, we must enable a first draw by having 
$\sum_{i=1}^k X_{0,i}\ge s$. 
Let $\tau_n$ be the total number of balls after $n$ draws.\footnote{Tenability adds additional conditions on the structure, which we discuss later.} 
For a balanced urn, 
we have
$$\tau_n = \tau_{n-1} + b = bn + \tau_0.$$
To study $(k,s,b)$-urns, we first need
to set up some notation.
\subsection{Notation} 
\label{subsec:Notation3}
We use the notation $\indicator_{\mathcal E}$ for the indicator of event $\mathcal E$.
For integer $r\ge 0$ and  $z\in \mathbb C$, the falling factorial $z(z-1) \cdots(z-r+1)$ is
denoted by $(z)_r$, and $(z)_0$ is to be interpreted as 1.
For a row vector $\vecs = (s_1, \ldots, s_k)$, {with $s_i\ge 0$ (for $i\in [k]$) and} $\sum_{i=1}^k s_i = s$,
the multinomial coefficient $s \choose s_1, \ldots, s_2$ may be written as $s \choose \vecs$ , and the concatenation $(y_1)_{s_1}(y_2)_{s_2}\cdots(y_k)_{s_k}$ of falling factorials may be succinctly written
as $({\bf y})_\mathbf{s}$ for economy, where ${\bf y} = (y_1, \ldots, y_k)$.  
 
There is a multinomial theorem for falling factorials, which expands a falling factorial of a sum 
in terms of the falling factorials of the individual numbers in the summand. Namely,
for $\mathbf{z}=(z_1, \ldots, z_k)\in \mathbb{ C}^k$, the theorem states that
\begin{equation}
(z_1 +\cdots + z_k)_s = \sum_{\vecs} {s\choose \vecs} (\mathbf{z})_\mathbf{s}.
\label{Eq:falling}
\end{equation} 
{This expansion is 
known as the (multivariate) Chu-Vandermonde identity}; the reader can find it {in} classic books like~\cite{Bailey,Graham}.

For probabilistic convergence modes we use $\convD, \inprob, \almostsure$ to 
respectively denote convergence in distribution, 
convergence in probability, and
convergence almost surely. 

We denote a diagonal matrix with the provided entries $a_1, a_2, \dots, a_k,$ as 
\begin{align*}
\mathbf{diag}(a_1, a_2, \dots, a_k) =\small
\begin{pmatrix}
			a_1 & 0&\ldots&0 \cr
                               	0 &a_2&\ldots&0 \cr
			\vdots& \vdots&\ddots&\vdots \cr
                               	 0& 0&\ldots&a_k\cr
\end{pmatrix},
\label{eq:diag}
\end{align*}
with $\mathbf{diag}(\mathbf{v})$ representing the diagonal matrix with {diagonal} entries equal to the components of a vector $\mathbf{v}$. 

We employ the standard asymptotic notation $o$ and $O$. We also need them in a matricial sense, which will be indicated in boldface, so $\bfo(n)$ and $\bfO(n)$ indicate a matrix
of appropriate dimensions with all its entries being $o(n)$ or~$O(n)$.

In the ensuing text, unless otherwise stated, all vectors are  {row vectors} of~$k$ components and
all square matrices are $k\times k$. The identity matrix is denoted by $\matI$, 
the row vector of zeros is denoted by ${\bf 0}$ and the zero matrix is
denoted by~${\bf 0}$. Likewise, the vector of ones is denoted by ${\bf 1}$. 
 
We define the discrete simplex $\Delta_{k,s}$ to be the collection of all possible samples~$\mathbf{s}$, or more precisely
\begin{align*}
\Delta_{k, s} = \Big\{ \vecs = (s_1, \ldots, s_k)\, \Big|\,   s_i\ge 0 , \ \ i\in[k];
         \ \  \sum_{j=1}^k s_j = s \Big\}.
\end{align*}
We arrange the ball addition numbers  in a ${s+k-1\choose s} \times k$ {\em replacement} 
matrix $\matM = [m_{\vecs,j}]_{\vecs\in \Delta_{k,s},j\in [k]}$, 
where the entry $m_{\vecs,j}$ is the number of balls of color $j$ that we add to the urn 
upon drawing the sample $\vecs$. We 
take the entire row corresponding
to $\vecs$ as a row vector denoted by $\vecm_\vecs =(m_{\vecs, 1}, 
m_{\vecs,2}, \dots ,a_{\vecs,k})$. 

When we deal with the eigenvalues $\lambda_1, \ldots, \lambda_k$ of a $k\times k$ matrix, we arrange them in decreasing order of their real parts. In other
words, we consider them indexed in the fashion
$$\Re \, \lambda_1 \ge \Re \, \lambda_2\ge \cdots\ge  \Re\, \lambda_k.$$
In the case $\lambda_1$ is of multiplicity 1, we call it the {\em principal eigenvalue}.
The associated eigenvector $\vecv_1$ is dubbed  {\em principal {left} eigenvector}.

The replacement matrix $\matM = [\vecm_\vecs]_{\vecs\in \Delta_{k,s}}$ is obtained by stacking up the row vectors $\vecm_\vecs$, for $\vecs\in\Delta_{k,s}$ from canonically smallest to largest (that is the smallest canonical sample labels the bottom row of $\matM$ 
and the largest canonical sample labels the top row of $\matM$).
\begin{example}
\label{Ex:MMatrix3}
\end{example}
Consider a $(3,3,9)$-urn scheme. The replacement matrix shown below is $10\times 3$, filled 
with some numbers consistent with the chosen parameters--- the sum across each row is the balance factor $b=9$.
The rows are labeled to the left with the partitions of $s=3$ into $k=3$ components listed 
(from top bottom) in decreasing canonical order.
For example, for $\vecs =(2, 0,1)$, the entry $m_{\vecs, 3}$ is 5.
 \begin{align*}
 \matM = \begin{matrix}
 300 \cr
                               	210 \cr
			201 \cr
                               	120 \cr
			111 \cr
                               	102 \cr
			030 \cr
			021 \cr
                               	012 \cr
			003 
\end{matrix}
\begin{pmatrix}3 & 3&3 \cr
                               	4 &2&3 \cr
			2 & 2&5 \cr
                               	5 & 1&3 \cr
			3 & 1&5 \cr
                               	1 & 1&7 \cr
			6 & 0&3 \cr
			4 & 0&5 \cr
                               	2 & 0&7 \cr
			0 & 0&9 
\end{pmatrix}{.}
\end{align*}
\subsection{Tools from linear algebra}
For a square matrix $\matA$,
we let $\sigma(\matA)$ be the spectrum of $\matA$ (its set of eigenvalues). We use the Jordan decomposition of $\matA$ to produce a direct sum $\bigoplus_\lambda \mathbf{E}_\lambda$ of general eigenspaces $\mathbf{E}_\lambda$ that decomposes $\mathbb{C}^k$, where $\matA-\lambda \matI$ is a nilpotent operator on $\mathbf{E}_\lambda$ for all $\lambda \in \sigma(\matA)$.

Using \cite{Nomizu} (Theorem 7.6), this result implies that for each $\lambda \in \sigma(\matA)$, there exist{s} a projection matrix $\mathbf{P}_\lambda$ that commutes with $\matA$ and satisfies the conditions
\begin{align*}
\sum_{\lambda \in \sigma(\matA)} \mathbf{P}_\lambda = \matI,
\end{align*}
\begin{align*}
\matA\mathbf{P}_\lambda = \mathbf{P}_\lambda \matA = \lambda \mathbf{P}_\lambda + \mathbf{N}_\lambda, 
\end{align*}
where $\mathbf{N}_\lambda = \mathbf{P}_\lambda \mathbf{N}_\lambda=\mathbf{N}_\lambda \mathbf{P}_\lambda$ is nilpotent. Furthermore, these projection matrices satisfy $\mathbf{P}_{\lambda_i} \mathbf{P}_{\lambda_j} = \mathbf{0}$, 
when $\lambda_i \ne \lambda_j$. We define $\nu_\lambda \ge 0$ to be the integer such that $\mathbf{N}_\lambda^{\nu_\lambda} \ne \mathbf{0}_{k \times k}$ but $\mathbf{N}_\lambda^{\nu_\lambda+1} = \mathbf{0}$. 

\begin{remark}
In the Jordan normal form of $\matA$, the largest Jordan block with~$\lambda$ has size $(\nu_\lambda+1)$ and thus $(\nu_\lambda+1)$ is at most equal to the multiplicity of $\lambda$. If the algebraic multiplicity of $\lambda$ is ${\cal M}_\lambda$, we can determine the size of Jordan blocks for~$\lambda$ by analyzing the dimension of the nullity of $(\matA-\lambda \matI)^i$ for $i=1, \dots,{\cal M}_\lambda$. If $\nu_\lambda=0$, for all $\lambda \in \sigma(\matA)$, then $\matA$ is diagonalizable.
\end{remark}
\subsection{The affine subclass}
Suppose $\field_n$ is the sigma field generated by the first $n$ {samples}. 
For a general scheme, the conditional expectation
$\E[\vecX_n \given \field_{n-1}]$ may contain combinations of $\vecX_{n-1}$ corresponding to higher moments of the counts in $\vecX_{n-1}$, such as $\E[\vecX_{n-1}^T \vecX_{n-1}]$,
of all mixed second moments. 
   
We call an urn scheme {\em affine}, if for each $n\ge 0$,
the conditional expectation of $\vecX_n$
takes the form
$$\E[\vecX_n\given \field_{n-1}] =  {\vecX_{n-1}  {\bf C}_n},$$
for some real-valued $k\times k$ matrix ${\bf C}_n$, thus keeping all unconditional
expectations linear. This condition aids the mathematical tractability of the model.

To guarantee affinity, certain conditions must be met. The following theorem specifies
necessary and sufficient conditions for affinity. These conditions are expressed in terms
of the vectors $\vece_1, \ldots, \vece_k$, which are the basis of a $k$-dimensional vector
space, where, for $i\in [k]$, we have 
	$$\vece_i = (\underbrace{0,0, \ldots 0}_{i - 1 \ \mbox{times}}, 1 , \underbrace{0,0, \ldots 0}
	      _{k - i \ \mbox{times}}){.}$$
 
The affinity will be built around the $k\times k$ submatrix of $\matM$ corresponding to the rows labeled $s\, \vece_i$, for $i \in [k]$. This submatrix is filled with a set of integers
that are balanced (that is, the sum across any row in the submatrix adds up to $b$). We call this square submatrix matrix the {\em core}, and denote it by~$\matA$; we refer to the $i^{th}$ row in  $\matA$ as
$\veca_{s\, \vece_i}$ (which is also $\vecm_{s\,\vece_i}$). To ensure tenability, we assume all the entries of 
$\matA$ are nonnegative, except possibly the diagonal elements---if one such element is negative, we assume it must be greater than or equal to $-s$. The initial number $\tau_{0}$ must be at least $s$, too. \footnote{This class may be expanded to cases where a diagonal element is less than $-s$, given specific starting conditions and column configurations for $\mathbf{A}$ are met.}

A version of the following theorem appears in~\cite{KubaX}. 
\begin{theorem}
\label{Thm:affine}
A $(k,s,b)$-urn with core
$\matA$  is affine, if and only if
$$\vecm_{\vecs} = \sum_{i=1}^k \frac {s_i} s\, \veca_{s\,\vece_i } = \frac 1 s \, \vecs \, \matA,
  \qquad  \mbox{for all \ }  
\vecs\in \Delta_{k,s}.$$
\end{theorem}
\begin{proof}
Recall that for a vector in the basis, 
$\vece_i = (e_1, \ldots, e_k )$, all the components are 0, except for $e_i$, which is 1. 
By construction, in the core we have the trivial relationship
\begin{align*}
\vecm_{s \,\vece_i}= \veca_{s \,\vece_i}
      =  \sum_{i=1}^k e_i \, \veca_{s\,\vece_i}= \vece_i \,\matA, \qquad  \mbox{for\ } i\in [k].
\end{align*} 
 
Let  the sample in the $n^\textnormal{th}$ draw be recorded in $\vecQ_n$.
The ball addition rules translate into a stochastic recurrence:
\begin{equation}
\vecX_n =  \vecX_{n-1} + \vecm_{\vecQ_n},
\label{Eq:Qn}
\end{equation}
giving rise to the conditional expectation
\begin{align*}
\E[\vecX_n\given \field_{n-1}] 
     &=  \vecX_{n-1} + \sum_{s\in \Delta_{k, s}} \vecm_\vecs\,
                  \Prob(\vecQ_n= \vecs\given \field_{n-1}) \\
     &=  \vecX_{n-1} + \sum_{{0 \le s_\ell,\, \ell\in [k]}  \atop{s_1+ \dots +s_k = s}} 
           \vecm_\vecs\, {\frac {{X_{n-1,1} \choose s_1}
             \cdots   {X_{n-1,k} \choose s_k}} 
          {{\tau_{n-1} \choose s}}.}
\end{align*}          

Assume the scheme is affine.
For affinity to hold,  $\E[\vecX_n\given \field_{n-1}]$ should be equated to 
{$\vecX_{n-1}\matC_n $},
for some $k\times k$ matrix $\matC_n$, for each $n\ge1$. This is the same 
as
\begin{align*} 
\vecX_{n-1}(\matC_n  - \matI)    (\tau_{n-1})_s
        =\sum_{{0 \le s_\ell,\, \ell\in [k]}  \atop {s_1+ \cdots +s_{k} = s}} 
        \vecm_\vecs\, {s \choose\vecs}
                   (X_{n-1,1})_{s_1} \cdots   (X_{n-1,k})_{s_k}.
\end{align*} 
Using the fact that $\tau_{n-1} = \sum_{i={1}}^k X_{n-1, i}$, 
write the latter relation in the form
\begin{align} 
&(X_{n-1,1}, \ldots, X_{n-1,k})(\mathbf{C}_n - \matI) \Big(\sum_{i=1}^k X_{n-1,i}\Big)_s \\
         &\qquad \qquad =   \sum_{{0 \le s_\ell,\, \ell\in [k]}  \atop {s_1+ \cdots +s_k = s}} 
     \vecm_\vecs\, {s \choose \vecs}
        (\vecX_{n-1})_\vecs. 
        \label{Eq:latter}
\end{align} 
On the left-hand, the highest power of $X_{n-1,i}$, for $i\in [k]$, is $s+1$, while 
that on the right-hand side is only $s$, unless $\mathbf{C}_n  - \matI$ rids the left-hand side
of these powers. So, $\mathbf{C}_n-\matI$ {\em can be} in the form $\frac 1 {\tau_{n-1}} \matB_n$,
for some matrix $\matB_n$ that does not depend on the history of composition vectors.
Write 
$$\matB_n = [\vecb_{n,1}^T, \ldots, \vecb_{n,k}^T],$$
where $\vecb_{n,j}^T$ is the $j^{th}$ {\em column} vector of $\matB_n$, for $j\in [k]$. 
So, we write the display~{(\ref{Eq:latter})} as
\begin{align*} 
\vecX_{n-1}{(\tau_{n-1} -1})_{s-1} \, [\vecb_{n,1}^T, \ldots, \vecb_{n,k}^T]
     =   \sum_{{0 \le s_\ell,\, \ell\in [k]}  \atop {s_1+ \cdots +s_{k} = s}} 
          \vecm_\vecs  {s \choose \vecs} (\vecX_{n-1})_\vecs. 
\end{align*} 

We utilize the multinomial theorem stated in~(\ref{Eq:falling}). 
For any $i\in k$, expand the falling factorial  ${(\tau_{n-1}-1)_{s-1}=\big(\big(\sum_{i=1}^k X_{n-1,i}
       \big) -1\big)_{s-1}}$ as
\begin{align*}
& (X_{n-1,1} +\cdots + X_{n-1,i-1} + (X_{n-1,i} -1) + X_{n-1,i+1} 
           + \cdots +X_{n-1,k})_{s-1}\\
       &\qquad\qquad = \sum_{{0 \le r_\ell,\, \ell\in [k]}  \atop {r_1+ \cdots +r_k= s-1}} 
    {s-1 \choose{\bf r}} (X_{n-1, i} - 1)_{r_\ell} \prod_{j=1\atop j\not = i}^s  (X_{n-1, j})_{r_j}.
\end{align*}    
This being valid for any $i\in [k]$, we get
\begin{align*}
& \vecX_{n-1} \sum_{{0 \le r_\ell,\, \ell\in [k]}  \atop {r_1+ \cdots +r_k= s-1}} 
    {s-1 \choose{\bf r}} (X_{n-1, i} - 1)_{r_\ell} \prod_{j=1\atop j\not = i}^s  (X_{n-1, j})_{r_j} \, [\vecb_{n,1}^T, \ldots, \vecb_{n,k}^T]\\
        & \qquad \qquad =   \sum_{{0 \le s_\ell,\, \ell\in [k]}  \atop {s_1+ \cdots +s_k = s}} 
     \vecm_\vecs {s \choose \vecs}
                (\vecX_{n-1})_\vecs. 
\end{align*} 
To match coefficients of similar falling factorials, execute the vectorial product 
on the left

$$\sum_{{0 \le r_\ell,\, \ell\in [k]}  \atop {r_1+ \cdots +r_k= s-1}} 
    {s-1 \choose{\bf r}} \sum_{i=1}^k(\vecX_{n-1})_{{\bf r} + \vece_i}\vecb_{n,i}^T
                   =   \sum_{{0 \le s_\ell,\, \ell\in [k]}  \atop { s_1+ \cdots +s_k = s}} 
                       \vecm_\vecs {s \choose \vecs}  (\vecX_{n-1})_\vecs.$$ 
Extracting the coefficient of $(\vecX_{n-1})_\vecs$ on both sides, we get
$$\sum_{i=1}^k {s-1 \choose \vecs - \vece_i} \vecb_{n,i}^T =    \vecm_\vecs {s \choose \vecs}.$$
Applying this to a row in the core, where {$\vecs= s\,\vece_j$, for some $j\in [k]$, 
we get $\vecb_{n,j}^T = s\,\vece_j$ (as all the multinomial coefficients
include negative lower indicies, except the $j^\textnormal{th}$)}. It follows that
$$ \vecm_\vecs = \frac {s_1! \, s_2!\, \cdots s_k!} {s!}\sum_{i=1}^k {s-1 \choose
       s_1,  \cdots s_{i-1} , s_i-1,  s_{i+1}, \ldots, s_k} \veca_{s\,\vece_i} = \sum_{i=1}^k
            \frac {s_i} s \, \veca_{s\,\vece_i} .$$

For the converse,  assume the affinity condition in the statement of the theorem.
Recall that $\vecQ_n$ is the random sample at time $n$, which has
a conditional multi-hypergeometric distribution   (given $\vecX_{n-1}$)
with parameters $\tau_{n-1}$, $\vecX_{n-1}$, and $s$. 
Now, we write the 
stochastic recurrence~(\ref{Eq:Qn}) in the form
\begin{equation}
\vecX_n  = \vecX_{n-1} + \frac 1 s  \,\vecQ_n \, \matA.
\label{Eq:Xnconditional}    
\end{equation}    
The mean of the multi-hypergeometric distribution is well known~\cite{Kendall}.
In vectorial form, the conditional mean (given $\vecX_{n-1}$) is
 $\frac s {\tau_{n-1}} \vecX_{n-1}$. 
We thus have the conditional expectation
\begin{equation}
\E[\vecX_n \given \field_{n-1}]  
    = \vecX_{n-1} + \frac 1 {\tau_{n-1}} \vecX_{n-1}\, \matA= \vecX_{n-1}\Big(\matI + \frac
              1 {\tau_{n-1}}\, \matA \Big) ,
              \label{Eq:cond}
\end{equation}
              which is an affine form.
\end{proof}
\begin{remark}
It is worthy to describe here $(k, s, b)$-urn schemes from a qualitative viewpoint. There is a saying that states ``the whole is equal to the sum of its parts.'' Here, each color plays an independent role on how the urn changes within a given step. That is, a ball of color 1 in the sample creates one set of changes, a ball of color 2 in the sample produces another set of changes, and so on. Thus, $ \vecm_\vecs $ is merely the sum of the separate impacts of each ball color, with no addition effect produced by a specific sample combination. Furthermore, this class simplifies to traditional study of single-draw urn models when $s=1$.
\end{remark}
\begin{cor}
\label{Cor:exactmean}
\begin{equation}
\bm{\mu}_n := \E[\vecX_n] = \vecX_0 \prod_{i=0}^{n-1} \Bigl(\matI + \frac{1}{\tau_i}
          \, \matA\Bigr).
\label{Eq:muAffine}
\end{equation}
\end{cor}
\begin{proof}
The corollary follows from direct iteration of~(\ref{Eq:cond}).
\end{proof}
\begin{example}
The matrix $\matM$ in Example~\ref{Ex:MMatrix3} is affine. It is built from the core
 \begin{align*}
 \matA=\small\begin{matrix}
           300 \cr
           030 \cr
		   003 
\end{matrix}
\begin{pmatrix}
             3 & 3&3 \cr
			6 & 0&3 \cr
			0 & 0&9 
\end{pmatrix}
       = \begin{matrix}
            3\vece_1 \cr
            3\vece_2 \cr
		   3\vece_3
\end{matrix}
\begin{pmatrix}
             3 & 3&3 \cr
			6 & 0&3 \cr
			0 & 0&9 
\end{pmatrix}
\end{align*}
As an instantiation of the computation of a row of $\matM$, as given
in Theorem~\ref{Thm:affine}, take the sample $\vecs =(2,0,1)$; 
the entry $\vecm_\vecs$ in $\matM$ is
obtained as
$$\vecm_{(2,0,1)} =\frac 2 3 \, \veca_{3\,\vece_1} + \frac 0 3 \, \veca_{3\,\vece_2} + \frac 1 3\,
              \veca_{3\, \vece_3} = \frac 2 3 (3, 3, 3)  +\frac 1 3(0, 0, 9) = (2,2,5).$$
\end{example}
\section{Classification of cores}
There is a large variety of $k\times k$ matrices that can be used as cores
to build replacement matrices for urn schemes. 
As already mentioned, our focus in this paper is on $(k,s,b)$-urn schemes.
Results are more easily obtained{,} when the core of the replacement matrix is irreducible. Loosely speaking,
this means that all the colors are essential (or communicating). More precisely,
this is defined in terms of dominant colors. Color $i\in [k]$ is said to be {\em dominant}, 
when we start with only balls of color $i$ and almost surely balls of color 
$j\in [k]$ appear at some future point in time, for every $j\not = i$. An urn scheme is then called irreducible, if every color in its
core is dominant. At the level of matrix theory, irreducibility means that $(\matI + \matA)^r$ is comprised of
nonnegative elements{,} for some $r\ge 0$. 

When an urn is irreducible,
the eigenvalue of $\mathbf{A}$ with the largest real part, $\lambda_1$, is real, positive and simple (of multiplicity 1), {a} consequence of a classic Perron-Frobenius theorem~(see~\cite{Horn}, Section 8.4).   

As in the sequel, the results for 
an affine irreducible $(k,s,b)$-urn scheme are governed 
by the ratio of $\Re(\lambda_2)$ to $\lambda_1$. We call this ratio the {\em core index} and denote it by $\Lambda$, which is always strictly less than 1 in irreducible core matrix. 

A trichotomy arises:
\begin{itemize}
\item [(a)] small-index core: where $\Lambda < \nicefrac 1 2$.
\item [(b)] critical-index core: where $\Lambda = \nicefrac 1 2$.
\item [(c)] large-index core: where $\Lambda > \nicefrac 1 2$.
\end{itemize}
Occasionally, we may just call the schemes small, critical, or large.

\section{A multivariate martingale underlying the $(k,s,b)$-urn class}
{The conditional expectation in~(\ref{Eq:cond})} leads to a martingale formulation.
We use the operator~$\nabla$ to denote backward differences. 
At draw~$n$, we have
\begin{align*}
\E\big[\nabla \mathbf{X}_n \,\big|\, \mathbb{F}_{n-1}\big]= \E\Big[\frac1 s\mathbf{Q}_{n} 
   \, \matA\, \Big|\, \mathbb{F}_{n-1}\Big]=\frac{1}{\tau_{n-1}}\mathbf{X}_{n-1}\, \matA.
\end{align*}

Define $\mathbf{Y}_n := \nabla \mathbf{X}_n - \frac{1}{\tau_{n-1}}\mathbf{X}_{n-1}\, \matA.$ 
Then, $\mathbf{Y}_n$ is $\mathbb{F}_{n-1}$-measurable and 
\begin{equation}
\E\big[\mathbf{Y}_n \,\big|\, \mathbb{F}_{n-1}\big]=\mathbf{0}.
\label{EQ:YZero}
\end{equation}
Thus, we have the recurrence
\begin{equation}
\mathbf{X}_n=\Big( \mathbf{I}+\frac{1}{\tau_{n-1}}\, \matA\Big)\mathbf{X}_{n-1} 
     + \mathbf{Y}_n.
\label{Eq:linearstochastic}
\end{equation}

In the proofs, we utilize the following matrix products
\begin{equation}
\label{EQ:FProd}
\mathbf{F}_{i,j}:=\prod_{i\le \ell <j} \Big(\mathbf{I} +\frac{1}{\tau_{\ell}}\, \mathbf{A}\Big), 
\qquad 0\le i<j,
\end{equation} 
to unravel the martingale differences in the process  $\mathbf{X}_n$. 
The product $\mathbf{F}_{i,j}$ can also be written {in terms of the polynomial function $f_{i,j}$}, where 
\begin{equation}
f_{i,j}(z):= \prod_{i\le \ell <j} \Big(1 +\frac{1}{\tau_{\ell}}z\Big)= \frac{\Gamma(j+\frac{1}{b}(\tau_0+z))}{\Gamma(j+\frac{\tau_0}{b})} \times\frac{\Gamma(i+\frac{\tau_0}{b}))}{\Gamma(i+\frac{1}{b}(\tau_0+z))}.
\label{Eq:polynomial}
\end{equation} 

We order the eigenvalues of $\mathbf{A}$ in the fashion specified in Section \ref{subsec:Notation3} and note that $\lambda_1={b}$ is a simple eigenvalue and 
$\Re(\lambda_i) < \lambda_1$, for $i > 1$, since $\mathbf{A}$ is irreducible. Accordingly, there exist corresponding left and right eigenvectors of 
$\mathbf{A}$, to be called $\mathbf{v}_1$ and $\mathbf{w}_1$, unique up to normalization. 
Since $\mathbf{A}$ is row-balanced, all the components of $\mathbf{w}_1$ are equal.
We may set 
$\mathbf{w}_1=\mathbf{1}$ and normalize $\mathbf{v}_1$ such that $\mathbf{w}_1\mathbf{v}_1^T=1$. This projection results in $\mathbf{P}_{\lambda_1}=\mathbf{w}_1^T \, \mathbf{v}_1$, and so, we have
\begin{equation}
\label{EQ:P1v}
\mathbf{v} \, \mathbf{P}_{\lambda_1} =\mathbf{v}\,\mathbf{1}^T\, \mathbf{v}_1, \textnormal{ for any vector }\mathbf{v} \in \mathbb{C}^k.
\end{equation}

The following lemmas provide important properties that can be applied to {the polynomial function~(\ref{Eq:polynomial})}. We omit the proofs as they are nearly identical to those in~\cite{Janson2020}.

\begin{lemma}
\label{Lemma:PFij}
For $1\le i\le j$ and $\lambda \in \sigma(\mathbf{A})$, we have

\begin{align*}
\mathbf{P}_\lambda \mathbf{F}_{i,j} =  \mathbf{  {O}} \Big( \Big(\frac{j}{i}\Big)^{\Re(\lambda)/b}\Big(1+\ln\Big( \frac{j}{i}\Big)\Big)^{\nu_\lambda} \Big).
\end{align*}
Furthermore, for any $\nu \ge \nu_\lambda$, as $i, j \to \infty$ with $i\le j$, we have
\begin{align*}
\mathbf{P}_\lambda \mathbf{F}_{i,j}=&\, \frac{1}{\nu!} \Big(\frac{j}{i}\Big)^{\lambda/b} \Big(\frac{\ln \big( \frac{j}{i}\big)}{  b}\Big)^{\nu} \mathbf{P}_\lambda  \mathbf{N}^\nu_\lambda 
+\mathbf{o} \Big( \Big(\frac{j}{i}\Big)^{\Re(\lambda)/b}\ln^{\nu}\Big( \frac{j}{i}\Big) \Big) \\
&\qquad {} +   \mathbf{  {O}} \Big( \Big(\frac{j}{i}\Big)^{\Re(\lambda)/b}\Big(1+\ln^{\nu-1}\Big( \frac{j}{i}\Big)\Big) \Big).
\end{align*} 

\end{lemma}
\begin{lemma}
When $\lambda_1=b$ is a simple eigenvalue, then for $0\le i \le j$,  we have
\begin{align}
\label{EQ:SimpleFP1}
\mathbf{P}_{\lambda_1}\, \mathbf{F}_{i,j} = f_{i,j} (\lambda_1)\, \mathbf{P}_{\lambda_1} = \frac{j+\frac{\tau_0}{b}}{i+\frac{\tau_0}{b}} \,\mathbf{P}_{\lambda_1}.
\end{align}
\end{lemma}
\begin{lemma}
\label{Lemma:FxA}
For any fixed $x \in (0,1]$, as $n\to \infty$, $\mathbf{F}_{\lceil xn \rceil, n}\to x^{-\frac{1}{  b}\mathbf{A}}$.
\end{lemma}

\subsection{Asymptotic mean}
While exact, the execution of the calculation in Corollary~\ref{Cor:exactmean} may be quite cumbersome.  
Here, we present a shortcut to the asymptotics of this computation via a multivariate martingale. 
\begin{theorem}
\label{Prop:F0P}
Suppose an affine $(k, s,b)$-urn scheme is built from an irreducible core $\matA$ with
principal eigenvector~$\vecv_1$ and core index $\Lambda\le 1$. Then, 
the expected value of the composition vector is 
$$\E[\mathbf{X}_n] = \lambda_1 n \vecv_1 + \tau_0 \vecv_1 +\bfO \big(n^\Lambda
      \ln ^{\nu_2}(n)\big).$$
\end{theorem}
\begin{proof}
The linear vectorial recurrence~(\ref{Eq:linearstochastic}) 
is of the general form
$${\bf L}_n = {\bf G_n} {\bf L}_{n-1} + {\bf H}_n,$$
with solution\footnote{Interpret the product corresponding to $i=n$ as the identity
matrix $\matI$.}
$${ \bf L}_n = \sum_{i=1}^n {\bf H}_i\prod_{j=i+1}^n {\bf G}_j
           + {\bf L}_0 \prod_{j=1}^n {\bf G}_j.$$

Write  the recurrence in the form $\vecX_n = \mathbf{F}_{n-1, n}+ \mathbf{Y}_n$, to
recognize the solution
\begin{equation}
\label{EQ:XF}
\mathbf{X}_n=\mathbf{X}_{0}\, \mathbf{F}_{0,n}+\sum_{i=1}^{n}\mathbf{Y}_i\, \mathbf{F}_{i,n}.
\end{equation}
The factors $\mathbf{X}_0$ and $\mathbf{F}_{i,j}$ are non-random.  We have
$\E\big[\mathbf{Y}_i\big]=0$, and by taking expectations we get
\begin{equation}
\label{EQ:XFMean}
\bm{\mu}_n=\E\big[\mathbf{X}_n\big]= \mathbf{X}_0\, \mathbf{F}_{0,n} = \mathbf{X}_0\, 
\, \matI \, \mathbf{F}_{0,n} 
= \sum_{\lambda \in \sigma(\matA)} \mathbf{X}_0\,  \mathbf{P}_\lambda\, \mathbf{F}_{0,n}. 
\end{equation}
Janson~\cite{Janson2020} analyzes the function $\mathbf{F}_{0,n}$ carefully.
When $\lambda \ne \lambda_1$, Lemma 4.3 in~\cite{Janson2020} implies that 
\begin{align*}
\mathbf{X}_0 \, \mathbf{P}_\lambda \, \mathbf{F}_{0,n} = \bfO\big(n^{\Re(\lambda)/b} \ln^{\nu_\lambda}(n)\big) = \bfO\big(n^{\Re(\lambda_2)/b} \ln^{\nu_2}(n)\big).
\end{align*}
By (\ref{EQ:P1v}), we get $\mathbf{X}_0\, \mathbf{P}_{\lambda_1}  
        = \mathbf{X}_0\, \mathbf{1}^T\, \mathbf{v}_1 = \tau_0\, \mathbf{v}_1$, and {so by (\ref{EQ:SimpleFP1})}, we have
\begin{align*}
 \mathbf{X}_0\, \mathbf{P}_{\lambda_1} \, \mathbf{F}_{0,n} = \frac{n+\frac{\tau_0}{b}}{\frac{\tau_0}{b}}  \mathbf{X}_0\, \mathbf{P}_{\lambda_1}=\frac{b n+\tau_0}{\tau_0}\, (\tau_0 \mathbf{v}_1) = (b n+\tau_0)\mathbf{v}_1.
\end{align*}
Therefore, we have 
$$\E [\mathbf{X}_n] =  (\lambda_1 n+\tau_0) \, \mathbf{v}_1+ \bfO \big(n^{\Lambda} \ln^{\nu_2}(n) \big) .$$
\end{proof}
\subsection{The covariance matrix}
For the covariance matrix, we define 
\begin{align*}
\hat{\mathbf{P}}  = \sum_{\lambda \ne \lambda_1}\mathbf{P}_{\lambda}= \mathbf{I} -\mathbf{P}_{\lambda_1}
\end{align*}
and the symmetric matrix 
\begin{equation}
\mathbf{B}:= \frac 1 {s^2}\, \mathbf{A}^T  \bm{\mathcal{Q}}\mathbf{A},
\label{EQ: BSym3}
\end{equation}
where $ \bm{\mathcal{Q}}= s(s-1) \mathbf{v}_1^T\mathbf{v}_1+s\, \mathbf{diag}(\mathbf{v}_1)$.
Lastly, we use $e^{{s} \mathbf{A}} := \sum_{k=0}^{\infty} \frac{1}{k!}\big({s}\mathbf{A}\big)^k $ for all $s\in \mathbb{R}$, to define the integral
\begin{equation}
\label{EQ:SigAInt3}
\mathbf{\Sigma}_{(\mathbf{A})}:=b \int_0^\infty e^{{s} \mathbf{A}^T}\,\hat{\mathbf{P}}^T\, 
\mathbf{B}\, \hat{\mathbf{P}}\,e^{{s} \mathbf{A}} e^{-{b} s}\, ds.
\end{equation}
\begin{theorem}
\label{Thm: CovAsymp}
Consider a growing irreducible affine $(k,s,b)$-urn. Then, for both unordered samples with and without replacement, and for large $n$, we have:

\begin{enumerate}
\item [\textnormal{(1)}] If the core index is small, then $\frac{1}{n}\bm{\Sigma}_n  \to
\bm{\Sigma}_{(\mathbf{A})}$, where $\bm{\Sigma}_{(\mathbf{A})}$ is a limiting matrix defined in (\ref{EQ:SigAInt3}).
\item [\textnormal{(2)}] If the core index is critical, then 
\begin{align*}
\frac 1 {n \ln^{2\nu_2+1}(n)}\, \bm{\Sigma}_n  \to  
\frac1{{b}^{2\nu_2}(2\nu_2+1)(\nu_2!)^2}
\sum_{\Re(\lambda) = {b}/2}(\mathbf{N}^*_\lambda)^{\nu_2}\,  \mathbf{P}^*_\lambda  \, \mathbf{B}\, \mathbf{P}_\lambda \, \mathbf{N}^{\nu_2}_\lambda.
\nonumber
\end{align*}
\end{enumerate}
\end{theorem}

Before we prove the theorems for the covariance, we provide some additional scaffolding. Using (\ref{EQ:XFMean}), we rewrite (\ref{EQ:XF}) as 
\begin{equation}
\label{EQ:DiffXMean}
\mathbf{X}_n -  \bm{\mu}_n = \sum_{i=1}^{n}  \mathbf{Y}_i \, \mathbf{F}_{i,n}.
\end{equation}
Thus, the covariance can be computed as 
\begin{align}
\Cov\big[\mathbf{X}_n] :=&\, \E \big[(\mathbf{X}_n-  \bm{\mu}_n)^T(\mathbf{X}_n-  \bm{\mu}_n)\big] \nonumber\\[1 em]
&= \E \Big[\sum_{i=1}^n \sum_{j=1}^n (\mathbf{Y}_i \mathbf{F}_{i,n} )^T ( \mathbf{Y}_j \mathbf{F}_{j,n}) \Big] \nonumber\\[1 em]
&= \sum_{i=1}^n \sum_{j=1}^n \mathbf{F}_{i,n}^T \,\E \big[\mathbf{Y}_i^T \mathbf{Y}_j \big] \mathbf{F}_{j,n}.
\label{EQ:VarSumIJ}
\end{align}
Since $\mathbf{Y}_j$ is {$\mathbb{F}_{j-1}$}-measurable, and
$\E\big[\mathbf{Y}_i \,\big|\, \mathbb{F}_j \big]=\mathbf{0}$,
when $i>j$,  we get 
\begin{align*}
\E\big[\mathbf{Y}_i^T\, \mathbf{Y}_j\big]=\E\big[\E\big[\mathbf{Y}_i^T \, \big|\, \mathbb{F}_j\big]\,\mathbf{Y}_j\big]= \mathbf{0}.
\end{align*}
In similar fashion, we have $\E\big[\mathbf{Y}_i^T\, \mathbf{Y}_j\big]= \mathbf{0}$ when $i<j$, and so (\ref{EQ:VarSumIJ}) reduces to
\begin{equation}
\label{EQ:VarFY}
\Cov\big[\mathbf{X}_n] = \sum_{i=1}^n \mathbf{F}_{i,n}^T \,\E \big[\mathbf{Y}_i^T\,
 \mathbf{Y}_i \big] \mathbf{F}_{i,n}.
\end{equation}
Since the urn is balanced, $\nabla \mathbf{X}_n \, \mathbf{1}^T=  {b}$ for all $n\ge1$. Thus, we have
\begin{align*}
\mathbf{Y}_n \mathbf{1}^T =\nabla \mathbf{X}_n \, \mathbf{1}^T - \E\big[\nabla \mathbf{X}_n \mathbf{1}^T\, \big| \, \mathbb{F}_{n-1} \big] =  {b}- {b}=0,
\end{align*}
which implies that $\mathbf{Y}_n \, \mathbf{P}_{\lambda_1}=\mathbf{0}$. Thus, we can rewrite (\ref{EQ:VarFY}) as
\begin{equation}
\label{EQ:VarFPY}
\Cov\big[\mathbf{X}_n\big]=\sum_{\lambda\in \sigma(\matA)} \sum_{ \varrho\in \sigma(\matA)} \sum_{i=1}^{n} \mathbf{F}_{i,n}^{T}\, \mathbf{P}_{\lambda}^T \,\E\big[\mathbf{Y}_{i}^T\, \mathbf{Y}_i \big]\, \mathbf{P}_\varrho\, \mathbf{F}_{i,n}.
\end{equation}
We define $\mathbf{T}_{i,n,\lambda,  \varrho}:=\mathbf{F}_{i,n}^{T} \, \mathbf{P}_{\lambda}^T \, \E\big[\mathbf{Y}_{i}^T\, \mathbf{Y}_i \big]\, \mathbf{P}_ \varrho \mathbf{F}_{i,n}$ for convenience. When either~$\lambda$ or $ \varrho$ is equal to $\lambda_1$, $\mathbf{P}_{\lambda}^T\, \E\big[\mathbf{Y}_{i}^T\, \mathbf{Y}_i \big]\, \mathbf{P}_ \varrho = \mathbf{0}$, 
{and so we may drop those cases and rewrite (\ref{EQ:VarFPY}) as}
\begin{equation}
\label{EQ:VarT}
\Cov\big[\mathbf{X}_n\big]=\sum_{\lambda \ne \lambda_1} \sum_{ \varrho \ne \lambda_1} \sum_{i=1}^{n} \mathbf{T}_{i,n,\lambda,  \varrho}.
\end{equation}
From this result, we produce three lemmas, the first of which gives us the general asymptotic growth of the covariances, the second demonstrating convergence in $\mathcal{L}_2$ and asymptotic approximation in $\mathcal{L}_1$, and the last providing us with the development of the $\mathbf{B}$ matrix described in (\ref{EQ: BSym3}).
\begin{lemma}
\label{VarCases}
If $\lambda_1$ is a simple eigenvalue, then for $n\ge 2$, {we have}
\begin{align}
\Cov \big[\mathbf{X}_n\big]=
\begin{cases}
 \mathbf{{O}}(n),& \mbox{for\ } \Lambda < \nicefrac 1 2;\\
 \mathbf{{O}}(n \ln^{2\nu_2+1}(n)),& \mbox{for\ } \Lambda = \nicefrac 1 2;\\
 \mathbf{{O}}(n^{2\Lambda} \ln^{2\nu_2}(n)), & \mbox{for\ }\Lambda > \nicefrac 1 2.
\end{cases}
\label{cov_big_oh}
\end{align}
Consequently, we have  $\Cov [\mathbf{X}_n]=\mathbf{o}(n^2)$, whenever
$\lambda_2 < \lambda_1$.
\end{lemma}
\begin{proof}
By construction, $\mathbf{Y}_n=\big(\frac{1}{s} \mathbf{Q}_n - \frac{1}{\tau_{n-1}}\mathbf{X}_{n-1}\big)\mathbf{A}$, and so 
\begin{align}
\label{EQ:Y2Q}
\E \big[\mathbf{Y}_n^T\mathbf{Y}_n\big]=\frac{1}{s^2}\, \mathbf{A}^T \E \big[\mathbf{Q}^T_n \mathbf{Q}_n\big] \mathbf{A} =  \mathbf{{O}}(1).
\end{align}
We provide detailed results of (\ref{EQ:Y2Q}) in Lemma \ref{Lemma:YYBvv}.
This result, along with Lemma \ref{Lemma:PFij}, implies that if $\lambda$ and $ \varrho$ are two eigenvalues of $\mathbf{A}$, then for $1\le i \le n$,
\begin{align*}
\mathbf{T}_{i,n,\lambda,  \varrho}
    &=(\mathbf{P}_{\lambda}\, \mathbf{F}_{i,n})^T \E\big[\mathbf{Y}_{i}^T\, \mathbf{Y}_i \big]\, \mathbf{P}_ \varrho\,  \mathbf{F}_{i,n}\\
    &= \mathbf{{O}} \Big( \Big(\frac{n}{i}\Big)^{(\Re(\lambda)+\Re( \varrho))/ {b}} \Big(1+ \ln\Big(\frac{n}{i}\Big) \Big)^{\nu_\lambda+\nu_ \varrho}\Big).
\end{align*}
If $\Re(\lambda)+\Re( \varrho) \ge  {b}$, we have that 
\begin{equation}
\label{EQ:TcritBound}
\mathbf{T}_{i,n,\lambda,  \varrho}= \mathbf{{O}} \Big( \Big(\frac{n}{i}\Big)^{(\Re(\lambda)+\Re( \varrho))/ {b}} \ln^{\nu_\lambda+\nu_ \varrho}(n)\Big).
\end{equation}
If $\Re(\lambda)+\Re( \varrho) <  {b}$, choose $\alpha$ such that $\frac{1}{ {b}}(\Re(\lambda)+\Re( \varrho)) <\alpha< 1$. Then, we are guaranteed that 
\begin{equation}
\label{EQ:TBound}
\mathbf{T}_{i,n,\lambda,  \varrho}= \mathbf{{O}} \Big( \Big(\frac{n}{i}\Big)^{\alpha} \Big).
\end{equation}
Summing together over $i$, we get
\begin{equation}
\label{EQ:TCases}
\sum_{i=1}^{n}\mathbf{T}_{i,n,\lambda,  \varrho}=
\begin{cases}
 \mathbf{{O}}(n),& \Re(\lambda) + \Re( \varrho) < {b} ;\\
 \mathbf{{O}}(n \ln^{\nu_\lambda+ \nu_ \varrho+1}(n)),& \Re(\lambda) + \Re( \varrho) = {b};\\
 \mathbf{{O}}(n^{(\Re(\lambda)+\Re( \varrho))/ {b}} \ln^{\nu_\lambda+\nu_ \varrho}(n)), &\Re(\lambda) + \Re( \varrho) > {b}.
\end{cases}
\end{equation}
Summing over all combinations of $\lambda,  \varrho \in \sigma(\mathbf{A})\backslash\{\lambda_1\}$,  $\sum_{i=1}^{n}\mathbf{T}_{i,n,\lambda,  \varrho}$ is of highest order when $\lambda= \varrho=\lambda_2,$ and so $\Cov \big[\mathbf{X}_n \big]=\mathbf{o}(n^2)$, when $\lambda_2 < \lambda_1$.
\end{proof}
\begin{lemma}
\label{Lemma:L12}
As $n\to\infty$, $\frac{1}{n}\mathbf{X}_n  \inTwo  {b} \mathbf{v}_1$. Furthermore, we have the asymptotic approximation
\begin{align*}
\mathbf{X}_n -  {b}n\mathbf{v}_1 =
\begin{cases}
 \mathbf{{O}}_{{\cal L}_1}(\sqrt{n}\,),&\mbox{for\ } \Lambda < \nicefrac{1}{2};\\
 \mathbf{{O}}_{{\cal L}_1}(n^{\nicefrac{1}{2}} \ln^{\nu_2+\nicefrac{1}{2}}(n)),& \mbox{for\ } \Lambda = \nicefrac{1}{2};\\
 \mathbf{{O}}_{{\cal L}_1}(n^{\Lambda} \ln^{\nu_2}(n)), &\mbox{for\ } \Lambda > \nicefrac{1}{2}.
\end{cases}
\end{align*}
\end{lemma}
\begin{proof}
From Lemma \ref{VarCases}, we have 
\begin{equation}
\Big|\Big|\frac{1}{n}\, \mathbf{X}_n -\frac{1}{n} \,\bm{\mu}_n\Big|\Big|_{{\cal L}_2}^2 = \frac{1}{n^{2}}||\mathbf{X}_n -  \bm{\mu}_n||_{{\cal L}_2}^2 
= \sum_{i=1}^k \frac{1}{n^2} \V \big[X_{n,i}\big] \to 0,
\end{equation}
and $\frac{1}{n}\, \bm{\mu}_n \to  {b}\mathbf{v}_1$ by Theorem \ref{Prop:F0P}. Thus, $||\frac{1}{n}\mathbf{X}_n -  {b}\mathbf{v}_1||_{{\cal L}_2}^2 \to 0$.
Now, given Lemma \ref{VarCases} and Theorem \ref{Prop:F0P}, for each $i \in [k]$, we obtain
\begin{align*}
	\E\big[\big(X_{n,i} -   {b}n v_{1,i} \big)^2\big] 
	   &= \E\big[\big(\big(X_{n,i} - \mu_{n,i}\big) 
 		+ \big(\mu_{n,i} -  {b}n v_{1,i}\big)\big) ^2\big]\\
	&=\V\big[X_{n,i}\big] + \big(\mu_{n,i} -  {b}n v_{1,i} \big)^2\\
    &=
\begin{cases}
O(n),&\mbox{for\ } \Lambda < \nicefrac{1}{2};\\
O(n \ln^{2\nu_2+1}(n)),& \mbox{for\ }\Lambda = \nicefrac{1}{2};\\
O(n^{2\Lambda} \ln^{2\nu_2}(n)), &\mbox{for\ } \Lambda > \nicefrac{1}{2}.
\end{cases}
	\label{Eq:Lonenorm}
	\end{align*}
So, by Jensen's inequality,
	$$\E\Big[\big | X_{n,i} -  {b}n v_{1,i} \big |\Big]
	\le \sqrt {\E\big[\big(X_{n,i}-   {b} v_{1,i} n  \big)^2\big] }\, .$$
It follows that
	$(\mathbf{X}_{n} - {b}n\mathbf{v}_1)$ has the asymptotic approximation claimed. 
\end{proof}
\begin{lemma}If $\lambda_1$ is a simple eigenvalue, then as $n\to\infty$, we have
\begin{align*}
\E\big[\mathbf{Y}_n^T\, \mathbf{Y}_n \big] \to \mathbf{B}- {b}^2 \, \mathbf{v}_1^T\, \mathbf{v}_1.
\end{align*}
Hence, for any eigenvalue $\lambda \ne \lambda_1$,
\begin{align*}
\E\big[\mathbf{Y}_n^T\, \mathbf{Y}_n\big]\mathbf{P}_\lambda \to \mathbf{B}\, \mathbf{P}_\lambda.
\end{align*} 
\label{Lemma:YYBvv}
\end{lemma}
\begin{proof}
Since $\mathbf{Y}_n=\nabla \mathbf{X}_n - \frac{1}{\tau_{n-1}} \, \mathbf{X}_{n-1}\, \mathbf{A}$ and $\mathbf{X}_{n}=\mathbf{X}_{n-1}+\frac{1}{s}\, \mathbf{Q}_n\, \mathbf{A}$,
\begin{align*}
\E \big[\mathbf{Y}^T_n \mathbf{Y}_n \,\big|\, \mathbb{F}_{n-1} \big] &=  
	\E \big[(\nabla \mathbf{X}_n)^T (\nabla\mathbf{X}_n) \,\big|\, \mathbb{F}_{n-1} \big]\\
	&\qquad\qquad {} - \E \big[(\nabla \mathbf{X}_n)^T (\tau_{n-1}^{-1}\, \mathbf{X}_{n-1}\, \mathbf{A}) \,\big|\, \mathbb{F}_{n-1} \big]\\
&\qquad\qquad {}-\E \big[(\tau_{n-1}^{-1}\, \mathbf{X}_{n-1}\, \mathbf{A})^T (\nabla\mathbf{X}_n) \,\big|\, \mathbb{F}_{n-1} \big]\\  
&\qquad\qquad {}+\E \big[(\tau_{n-1}^{-1}\, \mathbf{X}_{n-1}\, \mathbf{A})^T (\tau_{n-1}^{-1}\, \mathbf{X}_{n-1}\, \mathbf{A}) \,\big|\, \mathbb{F}_{n-1} \big] \\
&=  \E \big[(\nabla \mathbf{X}_n)^T (\nabla\mathbf{X}_n) \,\big|\, \mathbb{F}_{n-1} \big] \\
&\qquad\qquad {} - \E \big[(s^{-1}\mathbf{Q}_{n}\, \mathbf{A})^T \,\big|\, \mathbb{F}_{n-1} \big] (\tau_{n-1}^{-1}\, \mathbf{X}_{n-1}\, \mathbf{A}) \\
&\qquad\qquad {}-(\tau_{n-1}^{-1}\, \mathbf{X}_{n-1}\, \mathbf{A})^T \E \big[(s^{-1}\mathbf{Q}_{n}\, \mathbf{A}) \,\big|\, \mathbb{F}_{n-1} \big]\\
&\qquad\qquad {}+(\tau_{n-1}^{-2}\, \mathbf{X}_{n-1}\, \mathbf{A})^T 
 (\mathbf{X}_{n-1}\, \mathbf{A})\\
&= \E \big[(\nabla \mathbf{X}_n)^T (\nabla\mathbf{X}_n) \,\big|\, \mathbb{F}_{n-1} \big] - \tau_{n-1}^{-2}\, \mathbf{A}^T\mathbf{X}_{n-1}^T\, \mathbf{X}_{n-1}\, \mathbf{A}.
\end{align*}
Thus, $\E\big[\mathbf{Y}^T_n\, \mathbf{Y}_n \,\big|\, \mathbb{F}_{n-1} \big] =  \E \big[(\nabla \mathbf{X}_n)^T (\nabla\mathbf{X}_n)\big] - \tau_{n-1}^{-2}\, \mathbf{A}^T\E\big[\mathbf{X}_{n-1}^T
\mathbf{X}_{n-1}\, \big]\mathbf{A}.$

We tackle the asymptotics of this expression in two parts. First, note
\begin{equation}
\label{EQ:NablaQ}
	\E \big[(\nabla \mathbf{X}_n)^T (\nabla\mathbf{X}_n) \,\big|\, \mathbb{F}_{n-1} \big] 
	 = \frac 1 {s^2} \, \mathbf{A}^T  \E \big[\mathbf{Q}_{n}^T\, \mathbf{Q}_{n} \,\big|\, \mathbb{F}_{n-1} \big]\, \mathbf{A}. 
\end{equation} 
\begin{align*}
\E \big[\mathbf{Q}_{n}^T\, \mathbf{Q}_{n} \,\big|\, \mathbb{F}_{n-1} \big] =  \bm{\mathcal{Q}}_n,
\end{align*} 

with the entries
\begin{align}
 \bm{\mathcal{Q}}_n[i,j]=
\begin{cases}
\displaystyle{\frac{s(s-1)X_{n-1,i}^2}{\tau_{n-1}(\tau_{n-1}-1)}+\frac{s(\tau_{n-1}-s) X_{n-1,i}}{\tau_{n-1}(\tau_{n-1}-1)}},& i=j;\\
\\
\displaystyle{\frac{s(s-1)X_{n-1,i}X_{n-1,j}}{\tau_{n-1}(\tau_{n-1}-1)}},& i \ne j.
\end{cases}\label{Cov3QM}
\end{align}
{which follow from the known second moments of the multivariate hypergeometric distribution, see~\cite{Kendall}.}

Let $n_{\Lambda}$ be a function that represents the appropriate asymptotic order, 
given $\Lambda$ and Lemma \ref{Lemma:L12}. Rewrite $X_{n-1,i}= {b}(n-1)v_{1,i}+\bfO_{{\cal{L}}_1}(n_{\Lambda})$. Then, when $i=j$, we get 
\begin{align*}
\E \big[ \bm{\mathcal{Q}}_n[i,i]\big]=&\, \E\Big[\frac{s(s-1)( {b}(n-1)v_{1,i}+{O}_{{\cal{L}}_1}(n_{\Lambda}))^2}{\tau_{n-1}^{\underline 2}}\\ &\qquad\qquad\qquad+\frac{s(\tau_{n-1}-s) ( {b}(n-1)v_{1,i}+{O}_{{\cal{L}}_1}(n_{\Lambda}))}{\tau_{n-1}^{\underline 2}}\Big] \nonumber\\
&= s(s-1)v_{1,i}^2+s v_{1,i}\nonumber\\ 
&\qquad{} + \E\Big[s(s-1) \frac{2 {b}(n-1)v_{1,i}\,{O}_{{\cal{L}}_1}(n_{\Lambda})+{O}_{{\cal{L}}_1}(n_{\Lambda}^2)}{( {b}(n-1)+\tau_0)^2} \Big] \\
&\qquad {} + \E\Big[ \frac{{O}_{{\cal{L}}_1}(n_{\Lambda})}{ {b}(n-1)+\tau_0} \Big]\nonumber\\
&\to  \, s(s-1)v_{1,i}^2+s v_{1,i}, \qquad \mbox{as \ } n\to\infty.
\end{align*}
For $i\ne j$, we get, 
\begin{align*}
\E \big[ \bm{\mathcal{Q}}_n[i,j]\big] \to s(s-1) v_{1,i} v_{1,j}. \qquad  \mbox{as \ } n\to\infty,
\end{align*}
Define the matrix $ \bm{\mathcal{Q}}:= s(s-1)\mathbf{v}_1^T\, \mathbf{v}_1 + s\, \mathbf{diag}(\mathbf{v}_1).$ We take the expectation of (\ref{EQ:NablaQ}) to get
\begin{align*}
\E \big[(\nabla \mathbf{X}_n)^T (\nabla\mathbf{X}_n) \big] =  \frac 1 {s^2}\, \mathbf{A}^T  \E \big[ \bm{\mathcal{Q}}_{n}\big]\mathbf{A}
\to \frac 1 {s^2}\, \mathbf{A}^T   \bm{\mathcal{Q}}\, \mathbf{A} =:\mathbf{B}, 
\end{align*}
as provided by (\ref{EQ: BSym3}).
Since $\Cov \big[\mathbf{X}_n\big]=\mathbf{o}(n^2)$, we have that 
\begin{align*}
(\tau_0 +  {b}n)^{-2} \,\E\big[\mathbf{X}_{n-1}^T\mathbf{X}_{n-1}\big] 
&= (\tau_0 +  {b}n)^{-2}\, \V\big[\mathbf{X}_{n-1}\big]\nonumber\\
&\qquad \qquad {} +(\tau_0 +  {b}n)^{-2}\, \big(\E\big[\mathbf{X}_{n-1}\big]\big)^T \E\big[\mathbf{X}_{n-1}\big]\nonumber\\
&\to \, \mathbf{0}  + \mathbf{v}_1^T \mathbf{v}_1.
\end{align*}
Therefore, as $\mathbf{v}_1\mathbf{A}=\lambda_1 \mathbf{v}_1=  b\, \mathbf{v}_1$, 
\begin{align*}
\E\big[\mathbf{Y}_n^T \mathbf{Y}_n \big] \to \mathbf{B}-\mathbf{A}\mathbf{v}_1^T \mathbf{v}_1 \mathbf{A}=\mathbf{B}- {b}^2 \mathbf{v}_1^T \mathbf{v}_1.
\end{align*}
We complete the proof by noting that when $\lambda \ne \lambda_1$ we get that $\mathbf{v}_1 \mathbf{P}_\lambda=\mathbf{v}_1\, \mathbf{P}_{\lambda_1}\, \mathbf{P}_\lambda = \mathbf{0}.$
\end{proof}
With these results, we now prove Theorem \ref{Thm: CovAsymp}.
We use techniques quite similar to that found in \cite{Janson2020}, applying them to the multiple drawing scheme.
\begin{proof} (Theorem \ref{Thm: CovAsymp}, (1))\\
Let $\lambda,  \varrho \in \sigma(\mathbf{A})\backslash \{\lambda_1\}$. Then, $\Re(\lambda)$ and $\Re( \varrho)$ are at most $\Re(\lambda_2) <  {b}/2$. We convert the inner sum into an integral and get 
\begin{equation}
\frac{1}{n} \sum_{i=1}^{n} \mathbf{T}_{i,n,\lambda,  \varrho}= \int_0^1 \mathbf{T}_{\lceil xn \rceil,n,\lambda,  \varrho} \, dx.
\label{EQ:SumInt}
\end{equation}
For each fixed $x\in(0,1]$, Lemmas \ref{Lemma:FxA} and \ref{Lemma:YYBvv} imply that 
\begin{align*}
\mathbf{T}_{\lceil xn \rceil,n,\lambda,  \varrho} 
&=\mathbf{F}_{\lceil xn \rceil,n}^{T} \mathbf{P}_{\lambda}^T \, \E\big[\mathbf{Y}_{\lceil xn \rceil}^T\, \mathbf{Y}_{\lceil xn \rceil} \big]\, \mathbf{P}_ \varrho\, \mathbf{F}_{\lceil xn \rceil,n}
&\to x^{-\frac{1}{  b} \mathbf{A}^T}\mathbf{P}_{\lambda}^T\, \mathbf{B}\, \mathbf{P}_{ \varrho}\,x^{-\frac{1}{  b} \mathbf{A}}.
\end{align*}
Now, choose some $\alpha \in [0,1)$, such that $\frac{\Re(\lambda_2)}{ {b}}<\frac{\alpha}{2}$. Then, (\ref{EQ:TBound}) can be applied, providing that for some $C<\infty$, 
\begin{align*}
\mathbf{T}_{\lceil xn \rceil,n,\lambda,  \varrho} \le C\Big(\frac{n}{\lceil xn \rceil}\Big)^\alpha \le C x^{-\alpha},
\end{align*}
which is integrable over $(0,1]$. Applying Lebesgue Dominated Convergence to (\ref{EQ:SumInt}) and a change of variables to $x=e^{- {b}s}$, we get 
\begin{align*}
\frac{1}{n} \sum_{i=1}^{n} \mathbf{T}_{i,n,\lambda,  \varrho} 
   &\to \int_0^1 x^{-\frac{1}{  b} \mathbf{A}^T}\mathbf{P}_{\lambda}^T\, \mathbf{B}\, \mathbf{P}_{ \varrho}\,x^{-\frac{1}{  b} \mathbf{A}} \, dx \\
&= {b} \int_0^\infty e^{{s} \mathbf{A}^T}\mathbf{P}_{\lambda}^T\, \mathbf{B}\, \mathbf{P}_{ \varrho}\,e^{{s} \mathbf{A}} e^{- {b}s}\, ds.
\end{align*}
Thus, (\ref{EQ:VarT}) and the definition (\ref{EQ:SigAInt3}) give us the result
\begin{align*}
\frac{1}{n} \Cov \big[\mathbf{X}_n\big] =&\, \frac{1}{n} \sum_{\lambda \ne \lambda_1} \sum_{ \varrho \ne \lambda_1} \sum_{i=1}^{n}\mathbf{T}_{i,n,\lambda,  \varrho}\nonumber\\[1 em]
{}
&\to  b \int_0^\infty e^{{s} \mathbf{A}^T}\hat{\mathbf{P}}^T\mathbf{B}\, \hat{\mathbf{P}}\,e^{s \mathbf{A}} e^{- {b}s}\, ds\\
&=: \mathbf{\Sigma}_{(\mathbf{A})},
\end{align*}
as desired.
\end{proof}
\begin{proof} (Theorem \ref{Thm: CovAsymp}, (2))\\
We again use (\ref{EQ:VarT}) and consider $\sum_{i=1}^{n}\mathbf{T}_{i,n,\lambda,  \varrho}$ for eigenvalues $\lambda,  \varrho \in \sigma(\mathbf{A})\backslash \{\lambda_1\}$. By assumption, $\Re(\lambda)+\Re( \varrho) \le 2\Re(\lambda_2)= {b}$. If $\Re(\lambda)+\Re( \varrho) < {b}$, then we have $\sum_{i=1}^{n}\mathbf{T}_{i,n,\lambda,  \varrho}= \mathbf{ O}(n),$ based on (\ref{EQ:TCases}). Thus, we only need to consider cases where 
$\Re(\lambda)=\Re( \varrho) =\Re(\lambda_2)=\frac{ {b}}{2}.$ Furthermore, we have $\nu_\lambda, \nu_ \varrho \le \nu_2$. 

We again transform the sum into an integral, but first separate the term corresponding to $i=1$ from the sum and then use the change of variables $x=n^y = e^{y \ln(n)}$. From these steps, we get
\begin{align*}
\sum_{i=1}^{n}\mathbf{T}_{i,n,\lambda,  \varrho} = &\, \mathbf{T}_{1,n,\lambda,  \varrho}+\int_1^n \mathbf{T}_{\lceil x \rceil,n,\lambda,  \varrho} \, dx \nonumber\\[1 em]
{}
=&\,\mathbf{T}_{1,n,\lambda,  \varrho}+ \int_0^1 \mathbf{T}_{\lceil n^y \rceil,n,\lambda,  \varrho} \, n^y \ln(n) \, dy.
\end{align*}
From (\ref{EQ:TcritBound}), we know that $\mathbf{T}_{1,n,\lambda,  \varrho}= \mathbf{ O}\big(n \ln^{2\nu_2}(n)\big)$, and so
\begin{align*}
(n \ln^{2\nu_2+1}(n))^{-1}\sum_{i=1}^{n}\mathbf{T}_{i,n,\lambda,  \varrho} =\mathbf{o}(1)+ \int_0^1 \mathbf{T}_{\lceil n^y \rceil,n,\lambda,  \varrho} \, n^{y-1} (\ln(n))^{-2\nu_2} \, dy.
\end{align*}
We fix $y\in (0,1)$. Then, by Lemma \ref{Lemma:PFij},
\begin{align*}
\mathbf{P}_\lambda \mathbf{F}_{\lceil n^y \rceil,n}=&\, \frac{1}{\nu_2 !} \Big(\frac{n}{\lceil n^y \rceil} \Big)^{\lambda/ {b}} \Big(\frac{1}{  b} \ln \Big(\frac{n}{\lceil n^y \rceil} \Big)\Big)^{\nu_2} \big(\mathbf{P}_\lambda\,  \mathbf{N}^{\nu_2}_\lambda +\mathbf{o}(1) \big) \nonumber\\
&= \frac{1}{\nu_2 !} n^{(1-y)\lambda/ {b}}\Big(\Big(\frac{1-y}{ {b}}\Big) \ln(n)\Big)^{\nu_2} \big(\mathbf{P}_\lambda\, \mathbf{N}^{\nu_2}_\lambda +\mathbf{o}(1) \big) 
\end{align*}
We reach similar conclusions, when using the eigenvalue $ \varrho$.

Since $\Re(\lambda)+\Re( \varrho) =  b$, we define $\delta = \Im(\lambda)+\Im( \varrho)$, and so $\lambda+ \varrho= b+\delta \iu.$ Then, we have
\begin{equation}
\frac{n^{y-1}}{(\ln(n))^{2\nu_2}}\, \mathbf{T}_{\lceil n^y \rceil,n,\lambda,  \varrho} = \frac{1}{(\nu_2!)^2}n^{\frac{1}{  b}(1-y)\delta \iu}  \Big(\frac{1-y}{ {b}}\Big)^{2\nu_2} (\mathbf{N}^T_\lambda)^{\nu_2} \mathbf{P}_\lambda^T\, \mathbf{B}\, \mathbf{P}_ \varrho\,  \mathbf{N}^{\nu_2}_ \varrho + \mathbf{o}(1).
\label{EQ:Tnlog}
\end{equation}
From Lemma \ref{VarCases}, we get that for $y\in (0,1]$ and $n\ge 2$, 
\begin{align*}
\frac{n^{y-1}}{\ln^{2\nu_2}(n)}\, \mathbf{T}_{\lceil n^y \rceil,n,\lambda,  \varrho} 
     = \mathbf{O}\Big(\frac n {\lceil n^y\rceil}\Big)n^{y-1} = \mathbf{{O}}(1).
\end{align*}
Thus, the error bound $\mathbf{o}(1)$ in (\ref{EQ:Tnlog}) is also uniformly bounded, and so we apply Lebesgue Dominated Convergence to the new integral to obtain
\begin{align}
\label{EQ:CritInt1}
 \int_0^1 & \mathbf{T}_{\lceil n^y \rceil,n,\lambda,  \varrho} \, n^{y-1} (\ln(n))^{-2\nu_2} \, dy \nonumber\\
&=  \frac{1}{(\nu_2!)^2} \Big(\int_0^1  n^{\frac{1}{  b}(1-y)\delta \iu}  \Big(\frac{1-y}{ {b}}\Big)^{2\nu_2} \, dy\Big) (\mathbf{N}^T_\lambda)^{\nu_2} \mathbf{P}_\lambda^T \mathbf{B}^{(\mathbf {A})} \mathbf{P}_ \varrho \mathbf{N}^{\nu_2}_ \varrho+ \mathbf{o}(1).
\end{align}
If $\delta=0$ $(\bar{ \varrho}={\lambda})$, then the integral in (\ref{EQ:CritInt1}) simplifies to $ {b}^{-2\nu_2}(2\nu_2+1)^{-1}$. Furthermore, this situation yields $\mathbf{P}_\lambda=\mathbf{P}_{\bar{ \varrho}}=\bar{\mathbf{P}}_ \varrho$, and so $\mathbf{P}_\lambda^T = \mathbf{P}_ \varrho^*$. In similar fashion, we get $\mathbf{N}_\lambda^T=\mathbf{N}_ \varrho^*$, and thus (\ref{EQ:CritInt1}) becomes
\begin{align*}
   \frac 1 {\ln^{2\nu_2}(n)}
\int_0^1 \mathbf{T}_{\lceil n^y \rceil,n,  \bar{ \varrho}, \varrho} \, n^{y-1}  \, dy =\frac 1{ {b}^{2\nu_2}(2\nu_2+1)(\nu_2!)^2} (\mathbf{N}_ \varrho^*)^{\nu_2}\mathbf{P}_ \varrho^* 
\, \mathbf{B} \, \mathbf{P}_ \varrho\, \mathbf{N}_ \varrho^{\nu_2}+\mathbf{o}(1).
\end{align*}
If $\delta \ne 0$, then, {by} setting $u=1-y$, we have 
\begin{align*}
\int_0^1  n^{\frac{1}{  b}(1-y)\delta \iu}  \Big(\frac{1-y}{ {b}}\Big)^{2\nu_2} \, dy ={ {b}}^{-2\nu_2} \int_0^1  e^{(\frac{1}{  b}\delta \ln(n) \iu)u}  u^{2\nu_2} \, du \to 0,
\end{align*}
as $n\to\infty$, via integration by parts. Hence, when $\lambda \ne \bar{ \varrho}$, we get 
\begin{align*}  \frac 1 {\ln^{2\nu_2}(n)}
 \int_0^1 \mathbf{T}_{\lceil n^y \rceil,n,\lambda,  \varrho} \, n^{y-1} \, dy=\mathbf{o}(1).
\end{align*}
As mentioned earlier, we may ignore pairs where $\Re(\lambda)< {b}/2$ or $\Re( \varrho)< {b}/2$. We have now shown that cases of pairs such that $\Re(\lambda)=\Re( \varrho)= {b}/2$ but $\lambda \ne \bar{ \varrho}$ are asymptotically negligible as well. 
Therefore, we get 
\begin{align*}
\frac 1{n \ln^{2\nu_2+1}(n)} \, \Cov [\mathbf{X}_n] \to \frac 1
              {\lambda_1^{2\nu_2}(2\nu_2+1)(\nu_2!)^2}
\sum_{\Re(\lambda) = \lambda_1/2}(\mathbf{N}^*_\lambda)^{\nu_2} \mathbf{P}^*_\lambda\,   \mathbf{B} \, \mathbf{P}_\lambda \,  \mathbf{N}^{\nu_2}_\lambda,
\end{align*}
as desired.
\end{proof}

\subsection{A martingale multivariate central limit theorem}
\begin{theorem}
\label{Thm:AppAffMCLT}
Consider an  irreducible affine  $(k,s,b)$-urn scheme. As $n\to\infty$, we have
$$\frac 1 {\sqrt{\xi_n}} \, (\mathbf{X}_{n} - b  \mathbf{v}_1 n)\ \convD \ \normal_k(\mathbf{0},\bfsigma_\infty ),$$
where 
\begin{align*}
{\xi_n} =
\begin{cases}
n, &   \mbox{for\ }\Lambda < \nicefrac{1}{2};\\
n \ln^{2\nu_2+1}(n), & \mbox{for\ } \Lambda= \nicefrac{1}{2},
\end{cases}
\end{align*}
and
\begin{align*}
\bfsigma_\infty=
\begin{cases}
\bfsigma_{(\mathbf{A})}, &  \mbox{for\ } \Lambda <  \nicefrac{1}{2};\\
\frac{b^{-2\nu_2}}{(2\nu_2+1)(\nu_2!)^2}
\sum_{\Re(\lambda) = b/2}(\mathbf{N}^*_\lambda)^{\nu_2} \mathbf{P}^*_\lambda \,  \mathbf{B}\, \mathbf{P}_\lambda  \,\mathbf{N}^{\nu_2}_\lambda, &  \mbox{for\ }\Lambda=\nicefrac 1 2.
\end{cases}
\end{align*}
\end{theorem}
\begin{proof}
The proof draws from the construction of $\mathbf{Y}_n$. 
By construction and~(\ref{EQ:YZero}), we have that $\mathbf{Y}_i$ is a martingale difference sequence, and thus so is $\mathbf{Y}_i\, \mathbf{F}_{i,n}$. Furthermore, by  (\ref{EQ:DiffXMean}) we get that the sum of our martingale differences leads to $(\mathbf{X}_n - \bm{\mu}_n)$. 

To prove the conditional Lindeberg Condition, choose $\varepsilon>0$ and rewrite 
$\mathbf{Y}_i\, \mathbf{F}_{i,n}$ as 
\begin{align*}
\mathbf{Y}_i \, \mathbf{F}_{i,n} = (\vecX_i-\bm{\mu}_i) - (\vecX_{i-1}-\bm{\mu}_{i-1}).
\end{align*}

Then, we have
\begin{align}
\Big| \Big|\frac  1 {\sqrt{\xi_n}} \, \mathbf{Y}_i \mathbf{F}_{i,n} \Big| \Big|_{{\cal L}_2}
   &= \frac  1 {\sqrt{\xi_n}}\, \big| \big|(\mathbf{X}_i-\bm{\mu}_i) - (\mathbf{X}_{i-1}-\bm{\mu}_{i-1}) \big| \big|_{{\cal L}_2} \nonumber\\
&= \frac  1 {\sqrt{\xi_n}}\,\Big| \Big|\frac 1 s\mathbf{Q}_i\, \matA+{\mathbf{O}}
           \big(1+n^{\Re(\lambda_2)/b} \ln^{\nu_2}(n)\big) \Big| \Big|_{{\cal L}_2}   \nonumber\\
\to &\, 0, \label{EQ:L2Bound}
\end{align}
since {component-wise} we have $\mathbf{Q}_i \le s\mathbf{1}$. Based on (\ref{EQ:L2Bound}), there exists a natural number $n_0(\varepsilon)$, such that $\{||\mathbf{Y}_i\, \mathbf{F}_{i,n} ||_{{\cal L}_2} >\xi_n \varepsilon \}$ is empty for all $n>n_0(\varepsilon)$. Thus, we have that the summation
\begin{align*}
 \sum_{i = 1}^{n} \E \Bigg[\frac{1}{\xi_n}{\big(\mathbf{Y}_i\, \mathbf{F}_{i,n} \big)^T  (\mathbf{Y}_i\, \mathbf{F}_{i,n})} \, \indicator_
   {\big\{\big|\big|\xi_n^{-\frac{1}{2}} \,\mathbf{Y}_i \, \mathbf{F}_{i,n} \big|\big|_{{\cal L}_2} > \varepsilon\big\}}  \, \Bigg| \, \mathbb{F}_{i - 1} \Bigg] \almostsure 0,
\end{align*} 
which implies convergence in probability as well. Lindeberg's conditional condition
is verified.

By Lemma \ref{Lemma:YYBvv}, Theorem \ref{Thm: CovAsymp}, and based on the construction of $\bm{\mathcal Q}$, we have  
\begin{align*}
\sum_{i = 1}^{n} \E \Bigg[\frac 1 {\xi_n}  \big(\mathbf{Y}_i \, \mathbf{F}_{i,n} \big)^T 
 (\mathbf{Y}_i \, \mathbf{F}_{i,n}) \, \Bigg| \, \mathbb{F}_{i - 1} \Bigg] \inL{1} {\bm \Sigma}_\infty ,
\end{align*}
which implies convergence in probability. By {the} Martingale Central Limit Theorem (applied from \cite{Hall}, page 57--59), we get 
$$\frac 1 {\sqrt{\xi_n}} \sum_{i = 1}^{n}\mathbf{Y}_i \, \mathbf{F}_{i,n} = \frac 1 {\sqrt{\xi_n}} 
        (\mathbf{X}_{n}-\bm{\mu}_n)\  \convD  \ \normal_k (\mathbf{0},\bm{\Sigma}_\infty).$$

Applying Slutsky Theorem~\cite{Slutsky} to  $\xi_n^{-\frac{1}{2}}\mathbf{{O}}
\big(1+n^{\Re(\lambda_2)/b} \ln^{\nu_2}(n)\big)$, we get the result:
\begin{align*}
 \frac 1 {\sqrt{\xi_n}} \big(\mathbf{X}_{n}- b \mathbf{v}_1 n\big)\  \convD \, \normal_k (\mathbf{0},\bm{\Sigma}_\infty).
\end{align*}
\end{proof}

We conclude this section with a strong law for the composition vector.
The details for the proof are identical to that found on page 19 of \cite{Janson2020}.
\begin{theorem}
Consider an affine irreducible $(k,s,b)$-urn scheme. As $n\to\infty$, we have $$\frac{1}{n}\mathbf{X}_n \almostsure \lambda_1 \mathbf{v}_1.$$
\end{theorem}
\section{Comments on large-index urns}
As mentioned in \cite{Janson,Janson2020,Chauvin} and others, $\mathbf{X}_n$ possesses an almost-sure expansion based on spectral decomposition of $\mathbf{A}$. Unfortunately, deriving the covariance for large-index urns becomes trickier, as it may depend on the initial condition of the urn. 

However, by using the recursion (\ref{Eq:Xnconditional}), we have a recursive relationship for the covariance that may be used regardless of core index size.
\begin{cor}
\label{cor:lincovUO}
For a linear $k$-color urn model with unordered multiple drawings, the covariance matrix 
$\mathbf{\Sigma}_n = \E\big[(\mathbf{X}_n-\bm{\mu}_{n})^T(\mathbf{X}_n-\bm{\mu}_{n})\big]$
satisfies the following recurrence relation:
\begin{align*}
\mathbf{\Sigma}_n =& \, \Bigl( \mathbf{I} +\frac{1}{\tau_{n-1}}\mathbf{A}\Bigr)^{T} \mathbf{\Sigma}_{n-1}  \Bigl ( \mathbf{I} +\frac{1}{\tau_{n-1}}\mathbf{A}\Bigr)\\
&-\Big(\frac{\tau_{n-1}-s}{s \tau_{n-1}^2(\tau_{n-1}-1)}\Big) \mathbf{A}^{T} \bigl(\mathbf{\Sigma}_{n-1} +\bm{\mu}_{n-1}^T\bm{\mu}_{n-1}-\tau_{n-1}\,\mathbf{diag}(\bm{\mu}_{n-1})\bigr)\mathbf{A}.
\end{align*}
\end{cor}
\begin{proof}
We first derive a recursion for $\E\big[\mathbf{X}_n^T \mathbf{X}_n \, \big|\, \mathbb{F}_{n-1}\big]$ and then compute $\mathbf{\Sigma}_n = \E\big[\mathbf{X}_n^T \mathbf{X}_n\big]-\bm{\mu}_{n}^T\bm{\mu}_{n}$.

To begin, consider $\E\big[\mathbf{X}_n^T \mathbf{X}_n \, \big|\, 
\mathbb{F}_{n-1}\big]$. 
Using (\ref{Eq:Xnconditional}), we have
\begin{align}
\label{Cov3Rec}
\E \big[\mathbf{X}^T_n \mathbf{X}_n \,\big|\, \mathbb{F}_{n-1} \big] =&\,  
	\E \big[(\mathbf{X}_{n-1}+\frac{1}{s}\mathbf{Q}_n \mathbf{A})^T (\mathbf{X}_{n-1}+\frac{1}{s}\mathbf{Q}_n \mathbf{A}) \,\big|\, \mathbb{F}_{n-1} \big]\nonumber\\[1 em]
{}
=&\,  \mathbf{X}_{n-1}^T \mathbf{X}_{n-1} + \frac{1}{\tau_{n-1}} \big(\mathbf{A}^T \mathbf{X}_{n-1}^T \mathbf{X}_{n-1} + \mathbf{X}_{n-1}^T \mathbf{X}_{n-1}\mathbf{A} \big)\nonumber\\
&+ \frac{1}{s^{2}} \mathbf{A}\,\E \big[\mathbf{Q}_n^T\mathbf{Q}_n \,\big|\, \mathbb{F}_{n-1} \big]\, \mathbf{A}.
\end{align}

We then substitute (\ref{Cov3QM}) into (\ref{Cov3Rec}) to attain
\begin{align*}
\E \big[\mathbf{X}^T_n \mathbf{X}_n \,\big|\, \mathbb{F}_{n-1} \big] =&\, \Bigl( \mathbf{I} +\frac{1}{\tau_{n-1}}\mathbf{A}\Bigr)^{T} \mathbf{X}_{n-1}^T \mathbf{X}_{n-1} \Bigl ( \mathbf{I} +\frac{1}{\tau_{n-1}}\mathbf{A}\Bigr) \nonumber \\ 
& -\Big(\frac{\tau_{n-1}-s}{s \tau_{n-1}^2(\tau_{n-1}-1)}\Big)  \mathbf{A}^{T} \Big( \mathbf{X}_{n-1}^T \mathbf{X}_{n-1} -\tau_{n-1}\mathbf{diag} (\mathbf{X}_{n-1} ) \Big)\mathbf{A}.
\end{align*}

From here, we take the expected value of both sides and use 
$\E\big[\mathbf{X}_j^T \mathbf{X}_j\big]=\mathbf{\Sigma}_j + \bm{\mu}_{j}^T\bm{\mu}_{j}$ as well as relationship $\bm{\mu}_j = \bm{\mu}_{j-1} \big( \mathbf{I} + \frac{1}{\tau_{j-1}} \mathbf{A}\big)$ to get 
\begin{align*}
\mathbf{\Sigma}_n =&\, \Bigl( \mathbf{I} +\frac{1}{\tau_{n-1}}\mathbf{A}\Bigr)^{T} \big(\mathbf{\Sigma}_{n-1} + \bm{\mu}_{n-1}^T\bm{\mu}_{n-1}\big) \Bigl ( \mathbf{I} +\frac{1}{\tau_{n-1}}\mathbf{A}\Bigr) \\
&-\Big(\frac{\tau_{n-1}-s}{s \tau_{n-1}^2(\tau_{n-1}-1)}\Big) \mathbf{A}^{T} \bigl(\mathbf{\Sigma}_{n-1} +\bm{\mu}_{n-1}^T\bm{\mu}_{n-1}-\tau_{n-1}\mathbf{diag}(\bm{\mu}_{n-1})\bigr)\mathbf{A} \\ 
&- \Bigl( \mathbf{I} +\frac{1}{\tau_{n-1}}\mathbf{A}\Bigr)^{T}  \bm{\mu}_{n-1}^T\bm{\mu}_{n-1} \Bigl ( \mathbf{I} +\frac{1}{\tau_{n-1}}\mathbf{A}\Bigr)\\[1 em]
{}
=&\, \Bigl( \mathbf{I} +\frac{1}{\tau_{n-1}}\mathbf{A}\Bigr)^{T} \mathbf{\Sigma}_{n-1}  \Bigl ( \mathbf{I} +\frac{1}{\tau_{n-1}}\mathbf{A}\Bigr)\\
&-\Big(\frac{\tau_{n-1}-s}{s \tau_{n-1}^2(\tau_{n-1}-1)}\Big) \mathbf{A}^{T} \bigl(\mathbf{\Sigma}_{n-1} +\bm{\mu}_{n-1}^T\bm{\mu}_{n-1}-\tau_{n-1}\mathbf{diag}(\bm{\mu}_{n-1})\bigr)\mathbf{A}.
\end{align*}
\end{proof}

\section{Growth under sampling with replacement}
\label{Sec:rep}
If the sample is drawn with replacement, we get results that are very similar to the case 
of sampling without replacement, with very similar proof techniques. So,
we only point out the salient points of difference and very succinctly describe the main result.

To create a notational distinction between the without-replacement and the with-replacement sampling schemes, we use tilded variables to refer to the counterparts in the without-replacement schemes. 
Letting $\tilde{\mathbf{Q}}_n$ be the sample drawn with replacement at step $n$, 
we get multinomial distribution with conditional probability 
\begin{align*} 
\mathbb{P}\big(\tilde{\mathbf{Q}}_{n}=(s_1, \ldots, s_k)\,\big|\, \tilde \field_{n-1} \bigr)=&\, 
\begin{pmatrix} m \cr s_1, s_2, \dots, s_{k} \end{pmatrix} \frac{\tilde X_{n-1,1}^{s_1} \cdots 
             \tilde X_{n-1,k}^{s_k}}{\tau_{n-1}^{m}},
\end{align*}

The composition vector in this case has a mean value identical to that
in the case of sampling without replacement. The covariance matrix develops slightly differently from that of sampling without replacement, but remains of the same order as that described in (\ref{cov_big_oh}) and can be solved via Corollary \ref{cor:lincovUO}, but instead with $\frac{1}{s \tau_{n-1}^2}$ substituting 
for $\frac{\tau_{n-1}-s}{s \tau_{n-1}^2(\tau_{n-1}-1)}$.
Furthermore, for small- and critical-index urns, a central limit theorem follows identically to that of Theorem~\ref{Thm:AppAffMCLT}. 
\section{Examples}
We give four examples on $(k,s,b)$-urns. The first two  examples are
on small urns, one with a diagonalizable core and one with a non-diagonalizable core. 
We work out all the details in Example~\ref{Ex:Sm}, and portray a more sketchy picture in the rest.  
Example~\ref{Ex:Ht4} provides a useful application. Example~\ref{Ex:Crit} focuses on a critical urn, and 
we finish with a  
brief mention of the behavior of a large-index urn in Example~\ref{Ex:Large}. 
\begin{example}
\label{Ex:Sm} (Small diagonalizable core)
\end{example} 

Consider an affine urn scheme with $s=2$ draws per sample and irreducible core 
$$\mathbf{A}=\small\begin{pmatrix}
                                  6 & 4& 6 \\
                               	2 & 6&8 \\
		                        	4& 6&6
\end{pmatrix}.$$
This is a $(3,2,16)$-urn. Theorem~\ref{Thm:affine} completes the replacement matrix {to}
\begin{align*}
\matM = \begin{matrix}
     200 \\
     110 \\
     101\\
     020 \\
     011 \\
     002
\end{matrix}
\begin{pmatrix}
    6 & 4&6\\
    4 & 5&7\\
    5 &5&6 \\
   2 & 6&8\\
   3 & 6&7 \\
   4& 6&6 
\end{pmatrix}
\end{align*}

Suppose the urn starts in $\vecX_0= (4,3,5)$.
With $n=2000$, the computation in Corollary~\ref{Cor:exactmean} gives
$$\frac{1}{2000}\vecX_{2000} \approx (3.787, 5.527, 6.692).$$

The eigenvalues of $\mathbf{A}$ are $\lambda_1=16, \lambda_2 =1+\sqrt{5}$ and $\lambda_3 =1-\sqrt{5}$. With $\lambda_2 =1+\sqrt{5} < 8 = \lambda_1/ 2$, this is a small index case.
The principal left eigenvector is $\mathbf{v}_1 =(\frac{13}{55}, \frac{19}{55}, \frac{23}{55})$.
As $n \to \infty$, we have
\begin{align*}
\frac{1}{n}\E\big[\mathbf{X}_n\big] \to \bm{\mu}_\infty = 16\, \mathbf{v}_1
            = \Big(\frac{208}{55}, \frac{304}{55}, \frac{362}{55}\Big)
       \approx (3.782, 5.527, 6.691).
\end{align*}

We apply Theorem \ref{Thm: CovAsymp} to attain $\bm{\Sigma}_\infty$. We note that 
$\mathbf{A}$ is diagonalizable, and so we may write 
$\mathbf{A}= \mathbf{T}\, {\rm \bf diag} (\lambda_1, \lambda_2, \lambda_3)\, \mathbf{T}^{-1}$, where

$$\mathbf{T}=\small \frac{1}{2}\left(
\begin{array}{ccc}
 2 &-19 \sqrt{5}-43 & 19 \sqrt{5}-43 \\
 2 &13 \sqrt{5}+27 & 27 - 13\sqrt{5}\\
 2 & 2 & 2 \\
\end{array}
\right).
$$

Set $\mathbf{P}_{\lambda_1}=\mathbf{1}^T\, \mathbf{v}_1$, and 
$\mathbf{B}$ as in (\ref{EQ: BSym3}). Let 
$$e^{s \mathbf{A}}=\mathbf{T}\,\mathbf{diag}  (e^{16s}, \ \ 
  e^{(1+\sqrt{5}\,) s}, \ \ e^{(1+\sqrt{5}\, ) s})\,\mathbf{T}^{-1}.$$
Then,  we have
\begin{align*}
\bm{\Sigma}_{(\mathbf{A})} =16 \int_0^\infty e^{{s} \mathbf{A}^T}\, \hat{\mathbf{P}}^T\,
        \mathbf{B}\, \hat{\mathbf{P}}\, e^{{s} \mathbf{A}}\, e^{-16 s}\, ds.
\end{align*}
We apply Theorem~\ref{Thm:AppAffMCLT}, and conclude that
$$
\frac{1}{\sqrt{n}}\Big(\mathbf{X}_{n}-\Big(\frac{208}{55}, \frac{304}{55}, \frac{362}{55}\Big)\Big) \ 
        \convD \ \normal_3 \left(\mathbf{0},\ \frac{1}{3025}\small\begin{pmatrix}
 5552 & -2864 & -2688 \\
 -2864 & 1808 & 1056 \\
 -2688 & 1056 & 1632 \\
\end{pmatrix}\right).
$$
\begin{example}
\label{Ex:Ht4} (Small nondiagonalizable core)
\end{example}
A new series of weekly art events is about to begin in the city. You and three of your friends decide to go together to the kickstarter ceremony, but for the first official event, one friend cannot make it. To retain 
the group {size at} four, a new person is invited to join the group for the next event. For each subsequent {event}, this process repeats, where three of the prior attenders group together with a new person to attend. Former attenders may return to each new event, and the group can be formed via individuals who are linked by an outside acquaintance (so, the network does not necessarily need to include someone who knows all attending members).

To estimate the total number of individuals from the collective that have attended altogether one, two, three, and at least four of these events over a long-run, we may apply an affine urn model. Note that we begin with four individuals who have attended one event (the kickstarter), and with each new art event, we sample three former attendees to return with a new member. We always add one person to the group who has attended only one event (the new member), and if a former member is selected to attend the new event, they leave the list of those who attended $i$ events to the list of those who have attended $(i+1)$ events, unless they have attended at least four events, to which they remain in this category. Thus, the  affine model with $s=3$ drawings, has an initial state of $\mathbf{X}_0=(4,\,0,\, 0,\, 0)$ and the core matrix 
$$\mathbf{A}=\begin{pmatrix}-2 & 3& 0&0 \\
                               	1 & -3&3&0 \\
			1& 0&-3&3\\
					1& 0&0&0
\end{pmatrix}.$$
Here, we have an irreducible affine $(4,3,1)$-urn.
The eigenvalues of $\mathbf{A}$ are $\lambda_1=1$ and $\lambda_\ell=-3$ for $\ell\ge2$. 
With $\lambda_2 < \lambda_1 / 2$, this is a small urn. 
The leading left eigenvector is $\mathbf{v}_1=(\frac{1}{4}, \,\frac{3}{16},\, \frac{9}{64}, \,\frac{27}{64})$. Since $\mathbf{A}$ is 1-balanced, we have 
$\bm{\mu}_{\infty}=(\frac{1}{4}, \,\frac{3}{16},\, \frac{9}{64}, \,\frac{27}{64})$. 

Note that $\mathbf{A}$ is not diagonalizable, so the analysis requires the Jordan decomposition $\mathbf{A}= \mathbf{T}\, \bm{J}\, \mathbf{T}^{-1}$, where
$$\mathbf{T}={ 
\begin{pmatrix}
 1 & -3 & 1 & 0 \\
 1 & 1 & -\frac{4}{3} & \frac{1}{3} \\
 1 & 1 & 0 & -\frac{4}{9} \\
 1 & 1 & 0 & 0 \\
\end{pmatrix}}, ~~~~~~~\bm{{J}}={ \begin{pmatrix}
 1 & 0 & 0 & 0 \\
 0 & -3 & 1 & 0 \\
 0 & 0 &-3  &1 \\
 0 & 0 & 0 &-3
\end{pmatrix}}.$$ 

Again, we set $\mathbf{P}_{\lambda_1}=\mathbf{1}_k^T \mathbf{v}_1$, and $\mathbf{B}$ as in (\ref{EQ: BSym3}), and let 
$$e^{s \mathbf{A}}=\mathbf{T}\,{\small \begin{pmatrix}
 e^s & 0 & 0 & 0 \\
 0 & e^{-3 s} & e^{-3 s} s & \frac{1}{2} e^{-3 s} s^2 \\
 0 & 0 & e^{-3 s} & e^{-3 s} s \\
 0 & 0 & 0 & e^{-3 s} 
\end{pmatrix}}
\,\mathbf{T}^{-1}.$$

We apply Theorem \ref{Thm: CovAsymp}, to get the limiting matrix 
$\bm{\Sigma}_\infty$.  We apply Theorem~\ref{Thm:AppAffMCLT} to conclude that
\begin{align*}
&\frac{1}{\sqrt{n}}\Big(\mathbf{X}_{n}- \Big(\frac{1}{4}, \,\frac{3}{16},\, \frac{9}{64}, \,\frac{27}{64}\Big) n \Big)  \\
&\qquad\qquad \convD \, \normal_4 \left(\mathbf{0},
\begin{pmatrix}
 \frac{9}{112} & -\frac{207}{3136} & -\frac{2043}{87808} & \frac{783}{87808} \medskip\\
 -\frac{207}{3136} & \frac{11349}{87808} & -\frac{88983}{2458624} & -\frac{66501}{2458624}\medskip \\
 -\frac{2043}{87808} & -\frac{88983}{2458624} & \frac{7480413}{68841472} & -\frac{3387177}{68841472}\medskip\\
 \frac{783}{87808} & -\frac{66501}{2458624} & -\frac{3387177}{68841472} & \frac{4635333}{68841472} 
\end{pmatrix}\right).
\end{align*}
\begin{example} 
\label{Ex:Crit} (Critical core)
\end{example}
Consider an affine urn model with $s=2$ and replacement matrix with the core 
$$\mathbf{A}=\small\begin{pmatrix}4 & 0&2 \\
                               	2 & 4&0 \\
			0& 2&4
\end{pmatrix}.$$

The eigenvalues of $\mathbf{A}$ are $\lambda_1=6,  \lambda_2 = 3+\iu\sqrt{3}\, $, 
and $ \lambda_3 = 3-\iu\sqrt{3}\, $. Here the core index is $\nicefrac 1 2$, thus this urn is critical.
The principal left eigenvector is $\mathbf{v}_1 =\frac{1}{3}\, \mathbf{1}$. For large $n$, 
we have $\frac{1}{n}\E[\mathbf{X}_n] \to 6\, \mathbf{v}_1 = (2, 2, 2)$.  

Note that $\mathbf{A}$ is diagonalizable, and so $\nu_2=0$ and we have that  
{\begin{align*}
  \mathbf{P}_{\lambda_2} = &\, \frac{1}{6}\left(
\begin{array}{ccc}
 2 & -1-\iu \sqrt{3} & -1+\iu \sqrt{3} \\
 -1+\iu \sqrt{3} & 2 & -1-\iu \sqrt{3} \\
 -1-\iu \sqrt{3} & -1+\iu \sqrt{3} & 2 \\
\end{array}
\right),\\
\\
 \mathbf{P}_{\lambda_3} = &\, \frac{1}{6} \left(
\begin{array}{ccc}
 2 & -1+\iu \sqrt{3} & -1-\iu \sqrt{3} \\
 -1-\iu \sqrt{3} & 2 & -1+\iu \sqrt{3} \\
 -1+\iu \sqrt{3} & -1-\iu \sqrt{3} & 2 \\
\end{array}
\right)
\end{align*}}

By Theorem \ref{Thm: CovAsymp}, we get
\begin{align*}
\frac{1}{n \ln(n)} \bm{\Sigma}_n\to \bm{\Sigma}_{\infty}= \mathbf{P}^*_{\lambda_2}  \mathbf{B}\mathbf{P}_{\lambda_2}+ \mathbf{P}^*_{\lambda_3}  \mathbf{B}\mathbf{P}_{\lambda_3}.
\end{align*}
Upon applying Theorem~\ref{Thm:AppAffMCLT}, we conclude that
$$\frac{1}{n \ln(n)} (\mathbf{X}_{n}-(2,2,2)\, n)\  \convD \ \normal_3 \left(\mathbf{0}, \frac{1}{3}\small\begin{pmatrix}
 4 & -2 & -2 \\
 -2 & 4 & -2 \\
 -2 & -2 & 4 
\end{pmatrix}\right).
$$

\begin{example} 
\label{Ex:Large} (Large core)
\end{example}
Consider an affine urn scheme with $s=3$ draws per sample and irreducible core 
$$\mathbf{A}=\small
\begin{pmatrix}
             9 & 3&0\cr
			0 & 9&3 \cr
			3 & 0&9 
\end{pmatrix}$$
This is a $(3,3,12)$-urn. Suppose the urn starts in $\vecX_0= (3,2,2)$.
With $n=2000$, the computation in Corollary~\ref{Cor:exactmean} gives
$$\frac{1}{2000}\vecX_{2000} \approx {(3.992, 4.062, 3.947)}.$$

The eigenvalues of $\mathbf{A}$ here are $\lambda_1=12, \lambda_2 =7.5+3\iu\,\sqrt{\nicefrac{3}{4}}$ and $\lambda_3 =7.5-3\iu\,\sqrt{\nicefrac{3}{4}}$. With $ {\Re\, }\lambda_2 > \lambda_1/ 2$, this is a large index case.
The principal left eigenvector is $\mathbf{v}_1 =\frac{1}{3} \mathbf{1}$.
As $n \to \infty$, we have
\begin{align*}
\frac{1}{n}\E\big[\mathbf{X}_n\big] \to \bm{\mu}_\infty = 12\, \mathbf{v}_1
            = (4, 4, 4).
\end{align*}

\end{document}